\documentclass[11pt]{article}
\usepackage{a4, amsthm, amsmath, amssymb,amsbsy, epic,graphicx,mathrsfs,enumerate}
\usepackage[all]{xy}
\usepackage{multicol,longtable}

\newtheorem{thm}{Theorem}[section]
\newtheorem{lem}[thm]{Lemma}
\newtheorem{cor}[thm]{Corollary}
\newtheorem{prop}[thm]{Proposition}
\newtheorem{defi}[thm]{Definition}
\newtheorem{conj}{Conjecture}
\newtheorem{athm}{Theorem}
\newtheorem{alem}[athm]{Lemma}
\newtheorem{acor}[athm]{Corollary}
\newtheorem{aprop}[athm]{Proposition}

\theoremstyle{definition}
\newtheorem*{remark}{Remark}
\newtheorem{exam}{Example}[section]

\renewcommand{\ker}{\operatorname{ker}}               

\newcommand{\im}{\operatorname{im}}                   

\newcommand{\rank}{\operatorname{rank}}
\newcommand{\Aut}{\operatorname{Aut}}

\newcommand{\ra}{\rightarrow}
\newcommand{\mZ}{\mathbb{Z}}
\newcommand{\mF}{\mathbb{F}}
\newcommand{\mQ}{\mathbb{Q}}

\newcommand{\mN}{\mathbb{N}}

\newcommand{\Hom}{\operatorname{Hom}}

\newcommand{\Stab}{\operatorname{Stab}}

\newcommand{\Span}{\operatorname{span}}
\newcommand{\GL}{\operatorname{GL}}
\newcommand{\Tr}{\operatorname{Tr}}
\newcommand{\M}{\operatorname{M}}

\renewcommand{\rm}{\mathrm}

\newcommand{\mbf}{\mathbf}
\newcommand{\Cal}{\mathcal}
\newcommand{\End}{\operatorname{End}}
\newcommand{\hra}{\hookrightarrow}
\newcommand{\thra}{\twoheadrightarrow}

\renewcommand{\Tr}{\operatorname{Tr}}
\newcommand{\N}{\operatorname{N}}
\newcommand{\Fix}{\operatorname{Fix}}
\newcommand{\tr}{\operatorname{tr}}
\newcommand{\lda}{\lambda}

\newcommand{\bs}{\boldsymbol}
\newcommand{\scr}{\mathscr}
\newcommand{\lra}{\longrightarrow}

\newcommand{\ol}{\overline}
\newcommand{\preq}{\preccurlyeq}

\newcommand{\op}{\operatorname{op}}
\newcommand{\vph}{\vphantom}

\newcommand{\node}{\circle*{1.5}}

\newsavebox{\rvec}
\sbox{\rvec}{
\begin{picture}(20,3)
\put(-2,0){\vector(1,0){12}}
\put(10,0){\line(1,0){9}}
\end{picture}}

\newsavebox{\ruvec}
\sbox{\ruvec}{\begin{picture}(20,10)
\put(0,0){\vector(2,1){11}}
\put(10,5){\line(2,1){10}}
\end{picture}}

\newsavebox{\rdvec}
\sbox{\rdvec}{\begin{picture}(20,10)
\put(0,0){\vector(2,-1){11}}
\put(10,-5){\line(2,-1){10}}
\end{picture}}

\newsavebox{\rveca}
\sbox{\rveca}{\begin{picture}(10,3)
\put(0,0){\vector(1,0){6}}
\put(5,0){\line(1,0){5}}
\end{picture}}

\newsavebox{\ruveca}
\sbox{\ruveca}{\begin{picture}(10,5)
\put(0,0){\vector(2,1){6}}
\put(5,2.5){\line(2,1){5}}
\end{picture}}

\newsavebox{\rdveca}
\sbox{\rdveca}{\begin{picture}(10,5)
\put(0,0){\vector(2,-1){6}}
\put(5,-2.5){\line(2,-1){5}}
\end{picture}}

\title{Conjugacy classes in 
parabolic subgroups of general linear groups}
\author{Anton Evseev\footnote{Selwyn College, Cambridge, CB3 9DQ, UK, A.Evseev@dpmms.cam.ac.uk}}

\date{with an appendix by  \\
Anton Evseev and George Wellen\footnote{Mathematical Institute, 24-29 St Giles', Oxford, OX1 3LB, UK, wellen@maths.ox.ac.uk}}

\begin{document}

\maketitle

\abstract{We prove a formula connecting the number of unipotent conjugacy classes in a maximal parabolic subgroup of a finite general linear group with the numbers of unipotent conjugacy classes in various parabolic subgroups in smaller
dimensions. We generalise this formula and deduce a number of corollaries; in
particular, we express the number of conjugacy classes of unitriangular
matrices over a finite field in terms of the numbers of unipotent conjugacy
classes in maximal parabolic subgroups over the same field. We show how the
numbers of unipotent conjugacy classes in parabolic subgroups of small
dimensions may be calculated.}

\section{Introduction}\label{parintro}

Let $q$ be a prime power and $\mF_q$ be the finite field with $q$ elements. 
If $k$ and $m$ are nonnegative integers, let  
$\M_{m,k}(q)$ be the set of all $m\times k$ matrices over $\mF_q$.
Let $\mbf{l}=(l_1,\ldots,l_s)$ be a sequence of nonnegative integers.  
Writing  $\M_k (q) = \M_{k,k}(q)$, let $M^{\mbf l}(q)$ 
be the set of all matrices of the form
$$\begin{pmatrix}
\M_{l_1}(q) & \M_{l_1,l_2}(q) & \ldots & \M_{l_1,l_s}(q) \\
0 & \M_{l_2} (q) & \ldots & \M_{l_2,l_s}(q) \\
\vdots & \ddots & \ddots & \vdots \\
0 & 0 & \ldots & \M_{l_s}(q)
\end{pmatrix}.
$$

Let
$P^{\mbf{l}}(q)$ be the group of all invertible matrices in $\M^{\mbf l}(q)$. The group
$P^{\mbf{l}} (q)$ is called a \emph{parabolic} subgroup of the group 
$\GL_m (q)$, where $m=l_1+\cdots+l_s$. Alternatively, a parabolic subgroup may be described as the stabiliser of a flag in the vector space $\mF_q^m$.
Let $N^{\mbf l}(q)$ be the set of all nilpotent matrices in $M^{\mbf l}(q)$.

The group $P^{\mbf l}(q)$ acts on the set $M^{\mbf l}(q)$ by conjugation:
$^{g}x=gxg^{-1}$ ($g\in P^{\mbf{l}}(q)$, $x\in M^{\mbf{l}}(q)$). We investigate 
the number of orbits of this action and also the numbers of orbits of the 
actions of $P^{\mbf l}(q)$ on certain subsets of $M^{\mbf l} (q)$. 

Throughout the paper, we denote by $\gamma(G,X)$ the number of orbits of 
an action of a group $G$ on a finite set $X$, where the action is understood.
Let 
$$
\rho_{\mbf l}(q)=\gamma(P^{\mbf l}(q), N^{\mbf l}(q)).
$$

 Although the groups $P^{\mbf{l}}(q)$ have received a lot of attention, not much is known about the numbers $\gamma(P^{\mbf l}(q),P^{\mbf l}(q))$ and 
$\rho_{\mbf l}(q)$. 
In particular, it is not known whether the following is true.

\begin{conj} \label{gencon} For every tuple $\mbf l=(l_1,\ldots,l_s)$ of 
nonnegative integers, $\rho_{\mbf l}(q)$ is polynomial in $q$.
\end{conj}

(Here, and in what follows, we call a rational-valued function $f(x)$ defined on a set $D$ of integers \emph{polynomial} if there exists a polynomial $g$ 
with rational coefficients such that $f(x)=g(x)$ for all $x\in D$.) 
Write $(1^m)=\underbrace{(1,1,\ldots,1)}_{m}$.
The following two special cases of Conjecture~\ref{gencon} have received considerable attention and will be of particular interest to us.

\begin{conj}\label{unicon} For every positive integer $m$, 
$\rho_{(1^m)} (q)$ is a polynomial in $q$ with rational coefficients.
\end{conj}

\begin{conj}\label{maxcon} For any positive integers $k$ and $m$, 
$\rho_{(k,m)}(q)$ is a polynomial in $q$ with rational coefficients.
\end{conj}

It is not difficult to show that Conjecture \ref{unicon} is equivalent to the conjecture that the number of conjugacy classes of the group of upper unitriangular $n\times n$ matrices is a polynomial in $q$ with rational coefficients (see~\cite[Section 4.2]{thesis}). This last conjecture has been proved for $n\le 13$ by J.M. Arregi and A. Vera-L\'opez~\cite{pol1992, 
pol1995, pol2003}. Conjecture~\ref{maxcon} has been proved for 
$k\le 5$ (or, alternatively, for $m \le 5$); indeed, 
S.H. Murray~\cite{murray2000} proved that in those cases $\rho_{(k,m)}(q)$ 
does not depend on $q$.

A \emph{partition} is a (possibly, empty) non-increasing sequence of positive integers. 
If $\lambda=(\lambda_1,\dots,\lambda_r)$ is a partition, let 
$|\lambda|=\lambda_1+\cdots+\lambda_r$, and write $l(\lda)=r$. 
If $k,m\in \mN$, let 
$\Cal{P}_k$ be the set of all partitions $\lambda$ with $|\lda|=k$, let
$\Cal{P}^m$ be the set of partitions $\lambda$ such that $\lambda_i\le m$ 
for all $i$, and let  
$\Cal{P}_k^m= \Cal P_k \cap \Cal P^m$.
Let $p(k)=|\Cal{P}_k|$.

Let $m\in \mN$, and let 
$\lda=(\lda_1,\ldots,\lda_r)$ be a partition such that $\lda_1\le m$. 
Define
$$
\delta(m,\lambda)=(\lambda_r,\lambda_{r-1}-\lambda_r,
\lambda_{r-2}-\lambda_{r-1},\dots,\lambda_1-\lambda_2,m-\lda_1).
$$ 
Let $\nu^m_{\lda}(q)=\rho_{\delta(m,\lda)}(q)$. Thus,
$\nu^m_{\lda}(q)$ is the number of $P^{\mbf l}(q)$-orbits of $N^{\mbf l}(q)$ 
where $P^{\mbf l}(q)$ is the stabiliser of a   flag
$$
\mF_q^{\lda_r} \le \mF_q^{\lda_{r-1}} \le \cdots \le \mF_q^{\lda_1} \le \mF_q^m.
$$

We shall generalise the definition of $\nu_{\lda}^m (q)$ as follows.
Let $\overline{\mF}_q$ be the algebraic closure of the field $\mF_q$, 
and let $F$ be a subset of $\ol{\mF}_q$. Let $\Cal Y=\Cal Y_F$ be the    
family of all linear endomorphisms $T:U\ra U$ (where $U$ is an 
arbitrary finite dimensional vector space over $\mF_q$) 
such that all eigenvalues of $T$ 
over $\overline{\mF}_q$ are in $F$. 
We shall say that $\Cal Y$ is the \emph{class} 
of endomorphisms associated with $F$. 
We shall refer to elements of this class as \emph{$\Cal Y$-endomorphisms}.
Obviously, $\Cal Y$ is preserved by conjugation. 
Note that the family of all nilpotent endomorphisms and the 
family of all invertible endomorphisms are both classes. 
If $\Cal Y$ is a class and $m\in\mZ_{\ge 0}$, let $c(m, \Cal Y)$ 
be the number of $\GL_m (\mF_q)$-orbits on $\Cal Y\cap \M_m (\mF_q)$. 
We shall denote by $\Cal N$ the class of all nilpotent endomorphisms.
Note that 
$$
c(m,\Cal N)=p(m),
$$ 
the number of partitions of $m$. This follows
from the fact that nilpotent matrices in Jordan canonical form form a 
complete set of representatives of $\GL_m (q)$-orbits on $\N_m(q)$, the set
of nilpotent $m\times m$ matrices.

Let $\Cal Y$ be a class (of endomorphisms), and let $m\in \mN$. Let 
$\lda=(\lda_1,\ldots,\lda_r) \in \Cal P^m$ (that is, $\lda_1\le m$). 
Define 
$$
\kappa_{\lda}^m (\Cal Y)=
\gamma(P^{\delta (m,\lda)}(q),M^{\delta(m,\lda)}(q)\cap \Cal Y).
$$ 
One of the main results of this paper is as follows.
\begin{thm}\label{pargen} Let $k$ and $m$ be positive integers. 
Let $q$ be a prime power and $\Cal Y$ be a class of endomorphisms over $\mF_q$.
Then
$$\kappa_{(k)}^{k+m}(\Cal Y)=
\sum_{j=0}^{k} c(k-j, \Cal Y) \sum_{\lda\in \Cal P_j^m} \kappa_{\lda}^m (\Cal Y).
$$
In particular,
$$\rho_{(k,m)}(q)=
\sum_{j=0}^{k} p(k-j)\sum_{\lambda\in \Cal P_j^m} \nu_{\lda}^m (q).
$$
\end{thm}

\begin{remark} S.H. Murray~\cite{murray2006} has proved a similar result stated in terms of irreducible representations of parabolic subgroups. 
This result implies Theorem~\ref{pargen} in the case when $\Cal Y$ is the class of all invertible matrices.
The proof in~\cite{murray2006} is different from the one given here.
 Unlike the proof in this paper,
 the proof in~\cite{murray2006} establishes not just a numerical equality, 
but also an explicit correspondence between representations. 
It is an interesting question whether the 
method of~\cite{murray2006} can be extended to prove an analogue of the more 
general Theorem~\ref{gensq}. 
\end{remark}

We will deduce the following two results from Theorem \ref{pargen}.
\begin{thm}\label{impl} Let $m\in \mN$, and let $n=m(m-1)/2$. 
There exist integers 
$a_0,a_1,\dots,a_n$ such that, for all prime powers $q$,
$$\rho_{(1^m)}(q)=\sum_{j=0}^n a_j \rho_{(j,m)}(q).$$
Hence, Conjecture~\ref{maxcon} implies Conjecture~\ref{unicon}.
\end{thm}

If $\mbf l=(l_1,\ldots,l_s)$, write $\rho_{k,\mbf l}(q)$ for
$\rho_{(k,l_1,\ldots,l_s)}(q)$.

\begin{thm}\label{rec} Let $k, m\in \mN$, and let $n=m(m-1)/2$. 
There exist integers $a_{k0},\ldots,a_{kn}$ such that, for all
 tuples $\mbf l=(l_1,\ldots,l_s)$ of nonnegative integers with 
$l_1+\cdots+l_s=m$,
$$
\rho_{k,\mbf l} (q) = \sum_{j=0}^n a_{kj} \rho_{j,\mbf l}(q).
$$
\end{thm}

Using the methods developed in the proof of Theorem~\ref{pargen}, 
one may compute $\rho_{\mbf l}(q)$ for all tuples $\mbf l=(l_1,\ldots,l_s)$ 
with $l_1+\cdots+l_s\le 6$.
In particular, the following holds.

\begin{prop}\label{rhosmall} Let $\mbf l=(l_1,\ldots,l_s)$ be a tuple of nonnegative integers with $l_1+\cdots+l_s\le 6$. 
Then $\rho_{\mbf l}(q)$ is a polynomial in $q$ with positive 
integer coefficients. 
Hence (by Theorem~\ref{pargen}) $\rho_{(k,m)}(q)$ is polynomial 
in $q$ whenever $m\le 6$ and $k\in \mN$.
\end{prop}

This paper is organised as follows. In Section~\ref{quivers} 
we describe a general framework of quiver representations 
and their endomorphisms, which is used to state and prove the results. 
In Section~\ref{rednil} we show, in particular, how the numbers 
$\gamma(P^{\mbf l}(q),P^{\mbf l}(q))$ and $\gamma(P^{\mbf l}(q),M^{\mbf l}(q))$ 
may be expressed in terms of $\rho_{\mbf l'}(q^d)$ where $d\in \mN$ and 
$\mbf l'$ is another tuple of nonnegative integers.  This justifies our focus 
on the numbers $\rho_{\mbf l}(q)$. 

Section~\ref{parprem} contains a few standard results used later. 
In Section~\ref{dualact} we prove results that serve as the main tools 
allowing us to reduce problems such as that of counting $\rho_{(k,m)}(q)$ to 
other problems in a smaller dimension. In Section~\ref{indarg} we use
those tools to prove a generalisation of Theorem~\ref{pargen} stated in terms
 of quiver representations (Theorem~\ref{gensq}). 

In Section~\ref{inv} we invert the formulae in Theorems~\ref{pargen} 
and~\ref{gensq} and deduce Theorems~\ref{impl} and~\ref{rec}. 
This relies on combinatorial results proved jointly with G.~Wellen in the 
Appendix. I am very grateful to George Wellen for his part in this work.

In Section~\ref{presets} we describe a method for computing $\rho_{\mbf l}(q)$ 
when $m=l_1+\cdots+l_s$ is small. For this purpose, we generalise our problem
 to that of counting conjugacy classes of groups associated with preordered 
sets. Finally, Section~\ref{dualrep} investigates the symmetry afforded by 
considering a dual quiver representation. In particular, we show that
$$
\rho_{(l_1,\ldots,l_s)}(q)=\rho_{(l_s,\ldots,l_1)}(q).
$$

\textbf{Acknowledgments.} Most of this paper is a part of my D.Phil. thesis. I am very grateful to my supervisor, Marcus du Sautoy, and to George Wellen for
his part in this work. I would also like to thank my thesis examiners, Dan Segal and Gerhard R\"ohrle, for spotting a number of errors and for helpful comments.

\vspace{0.8cm} 
\textit{Notation and definitions}
\nopagebreak

\begin{itemize}
\item $\gamma(G,X)$ is the number of $G$-orbits on a finite set $X$ where the 
action of a group $G$ on $X$ is understood;
\item $[k,n]=\{k,k+1,k+2,\ldots,n\}$ where $k \le n$ are integers;
\item $|X|$ is the cardinality of a set $X$;
\item $\Cal A$ is a \emph{partition of a set} $X$ if $\Cal A$ is a family 
of disjoint sets whose union is $X$; two elements $x$ and $y$ of $X$ are 
said to be $\Cal A$-\emph{equivalent} if there exists $A\in \Cal A$ such that
$x,y\in A$;
\item $\delta_{ij}=0$ if $i\ne j$, and $\delta_{ii}=1$;
\item $A\sqcup B$ is the disjoint union of sets $A$ and $B$ 
(formally defined as $A\times \{0\} \cup B\times \{1\}$);
\item $I_V=I$ is the identity element of $\GL(V)$;
\item $I_k$ is the identity $k\times k$ matrix over 
an appropriate field.
\item Suppose $U$ and $V$ are vector spaces, $X\in \End(U)$ and $Y\in \End(V)$;
then $\Cal I(X,Y)$ denotes the vector space 
of all linear maps $T:U\ra V$ such that $TX=YT$;  
\item $\N_n (K)$ is the set of all nilpotent matrices in $\M_{n,n}(K)$;
\item 
$\End(V;U_1,\ldots,U_k)=\{ f\in \End(V): f(U_i)\subseteq U_i \; \forall i \}$ 
where $U_i$ are subspaces of $V$;
\item $\scr P(V;U_1,\ldots,U_k):=\End(V;U_1,\ldots,U_k) \cap \GL(V)$;
\item By convention, the set of $0\times k$ matrices contains just one element
 (which is nilpotent if $k=0$), and the group $\GL_{0}(K)$ is trivial;
\item If $I$ and $J$ are finite sets, then $\M_{I,J}(K)$ is the set of 
all matrices over a field $K$ whose rows are indexed by the elements of $I$ and whose
columns are indexed by elements of $J$; we refer to these
 as \emph{$I\times J$ matrices}; 
note that, if $A\in \M_{I,J}(K)$ and $B\in \M_{J,J'}(K)$, then the 
product $AB\in \M_{I,J'}(K)$ is well defined;
\item $A^{t}$ is the transpose of a matrix $A$;
\item $\Tr (A)$ is the trace of a square matrix $A$; 
\item If a group $G$ acts on a set $X$, two elements of $X$ are said to be 
$G$-\emph{conjugate} if they are in the same $G$-orbit;
\item All \emph{rings} are understood to have an identity element;
\item If $R$ is a ring, then $R^{\op}$ is the ring with the same underlying
abelian group and with the multiplication $(r,s)\mapsto sr$;
\end{itemize}

\section{Quiver representations and automorphisms}\label{quivers}

We recall the standard definitions related to quivers (see~\cite{ARS}, for example). A \emph{quiver} is a
pair $(E_0, E_1)$ of finite sets together with maps $\sigma:E_1\ra E_0$ and
$\tau: E_1 \ra E_0$. Elements of $E_0$ may be thought of as nodes; then each
element $e\in E_1$ may be represented as an arrow from $\sigma(e)$ to $\tau(e)$.
Let $K$ be a field. A \emph{representation} of a quiver $(E_0,E_1)$ over $K$ 
is a 
pair $(\mbf U, \bs\alpha)$ such that
\begin{enumerate}[(i)]
\item $\mbf U=(U_{a})_{a\in E_0}$ is a tuple of vector spaces over $K$;
\item $\bs\alpha=(\alpha_e)_{e\in E_1}$ where 
$\alpha_e\in \Hom(U_{\sigma(e)},U_{\tau(e)})$.
\end{enumerate}
If $(\mbf U,\bs\alpha)$ is a representation of a quiver $(E_0,E_1)$, we shall 
refer to the quadruple $Q=(E_0,E_1,\mbf U,\bs\alpha)$ as a \emph{quiver representation}. A quiver representation may be thought of as a collection of vector
spaces together with linear maps between some of those spaces.

If $(\mbf U,\bs\alpha)$ and $(\mbf U',\bs\alpha')$ are two representations 
of a quiver $(E_0,E_1)$, a \emph{morphism} between those representations 
is a tuple $(X_a)_{a\in E_0}$ such that
\begin{enumerate}[(i)]
\item $X_a\in \Hom(U_a,U'_a)$ for all $a\in E_0$;
\item $\alpha'_e X_{\sigma(e)} = X_{\tau(e)} \alpha_e$ for all $e\in E_1$.
\end{enumerate}
This defines the category of representations of $(E_0,E_1)$.

If $Q=(E_0,E_1,\mbf U,\bs\alpha)$ is a quiver representation, write
$\End(Q)$ for the ring of all endomorphisms of $Q$ and $\Aut(Q)$ for the group
of all automorphisms of $Q$.  Let $N(Q)$ be the set of all nilpotent 
elements of $\End(Q)$. The group $\Aut(Q)$ acts naturally on the set 
$\End(Q)$ by conjugation: if 
$\mbf{X}=(X_a)\in \End(Q)$ and $\mbf{g}\in \Aut(Q)$, then 
$$\mbf{g}\circ \mbf{X} =   
(g_a X_a g_a^{-1})_{a \in E_0}.
$$

Assume that $K=\mF_q$, where $q$ is a prime power.
We shall be concerned with the number of orbits of this action and, more generally, with the number of orbits of actions of certain subgroups of $\Aut(Q)$ on certain subsets of $\End(Q)$, such as $N(Q)$ for example. 
Note that $\Aut(Q)$ and $N(Q)$ may be defined in terms of the ring structure
on $\End(Q)$, so they are preserved by ring isomorphisms.
Let
$$
\theta(Q) = \gamma(\Aut(Q), N(Q)).
$$

All the problems discussed in the introduction may be stated in terms of quiver representations.
If $\mbf l = (l_1,l_2,\ldots,l_s)$ is a tuple of 
nonnegative integers with $m=l_1+\cdots+l_s$, let $R_{\mbf l}=R_{\mbf l}(q)$ 
be the quiver representation
$$
\xymatrix{
\mF_q^{l_1}\ar@{^{(}->}[r] & \mF_q^{l_1+l_2}\ar@{^{(}->}[r] & 
\cdots \ar@{^{(}->}[r] &
\mF_q^{l_1+\cdots+l_{s-1}}\ar@{^{(}->}[r] & \mF_q^m
}
$$
where all the arrows represent injective linear maps. More formally,
$$
R_{\mbf l}=([1,s],[1,s-1],\mbf U,\bs\alpha)
$$ 
with $\sigma(i)=i$, 
$\tau(i)=i+1$ for all $i\in [1,s-1]$, $U_i=\mF_q^{l_1+\cdots+l_i}$ and
$\alpha_i$ injective. Obviously, these conditions define $R_{\mbf l}$ up to
isomorphism of representations.

We may choose a basis $\{b_1,\ldots,b_m\}$ of $U_s=\mF_q^m$ so that, for 
each $i$, the image of $U_i=\mF_q^{l_1+\cdots+l_i}$ under 
$\alpha_{s-1}\cdots\alpha_{i+1}\alpha_i$ is equal to the span of 
$\{b_1,\ldots,b_{l_1+\cdots+l_i} \}$. Let 
$J_{\mbf l}: \End(R_{\mbf l}) \ra \M_{m,m}(q)$ be the map which assigns to each
$\mbf X=(X_i)_{i\in [1,s]}\in \End(R_{\mbf l})$ the matrix of 
$X_s\in \End(\mF_q^m)$ with respect to the basis $\{b_1,\ldots,b_m\}$. 
Then the following result is obvious.

\begin{lem}\label{quivpar} 
The map $J_{\mbf l}$ is a ring isomorphism from $\End(R_{\mbf l})$ 
onto $M^{\mbf l}(q)$. 
Hence, $\rho_{\mbf l}(q)=\theta(R_{\mbf l})$. 
\end{lem}  

Call a quiver representation $Q'=(E'_0,E'_1,\mbf U', \bs\alpha')$ 
an \emph{extension} of a quiver representation $Q=(E_0,E_1,\mbf U,\bs\alpha)$ 
if 
\begin{enumerate}[(i)]
\item $E'_0\supseteq E_0$;
\item $U'_a=U_a$ for all $a\in E_0$;
\item for every $\mbf X=(X_a)_{a\in E'_0}\in \End(Q')$, the tuple 
 $(X_a)_{a\in E_0}$ belongs to $\End(Q)$. 
\end{enumerate}

Let $Q'$ be an extension of $Q$. Define a map 
$\pi=\pi_Q^{Q'}:\End(Q') \ra \End(Q)$ by
$$(X_a)_{a\in E'_0} \mapsto (X_a)_{a\in E_0}.$$   
If $B\subseteq \End(Q)$, let $B^{Q'}=\pi^{-1}(B)$.
 If $G$ is a subgroup of $\Aut(Q)$, define
$G^{Q'}=\pi^{-1}(G)\cap \Aut(Q')$. (This will cause no ambiguity if a group is considered a distinct object from the set of its elements.)

Let $Q=(E_0,E_1,\mbf U, \bs\alpha)$ be a quiver representation, 
and let $e\in E$ with 
$\sigma(e)=a$, $\tau(e)=b$. Let $Y=\ker(\alpha_e)\le U_a$, 
$Z=\im(\alpha_e)\le U_b$.
Define an extension $K(Q,e)=(E'_0,E'_1,\mbf U',\bs\alpha')$ as follows:
\begin{enumerate}[(i)]
\item $E'_0=E_0 \sqcup \{ c \}$;
\item $U'_c= Y$ (and $U'_x=U_x$ for $x\in E_0$); 
\item $E'_1=E_1\sqcup \{ e' \}$ where $\sigma(e')=c$, $\tau(e')=a$;
\item $\alpha'_{f}=\alpha_f$ for all $f\in E_1$, and $\alpha'_{e'}$ is 
the inclusion map $Y\hra U_a$.  
\end{enumerate}
Define another extension $I(Q,e)=(E''_0,E''_1,\mbf U'',\bs\alpha'')$ 
as follows:
\begin{enumerate}[(i)]
\item $E''_0=E_0 \sqcup \{ d \}$;
\item $U''_d = Z$; 
\item $E''_1=(E_1\setminus \{ e \}) \sqcup \{ e^{\sharp}, e'' \}$,
where $\sigma(e'')=d$, 
$\tau(e'')=b$, $\sigma(e^{\sharp})=a$, $\tau(e^{\sharp})=d$;
\item  $\alpha''_{f}=\alpha_f$ for all $f\in E_1 \setminus \{ e \}$, 
 $\alpha''_{e''}$ is the inclusion map $Z\hra U_b$, 
and $\alpha''_{e^{\sharp}}$ is the map 
$U_a\ra Z$ given by $\alpha''_{e^{\sharp}}(v)=\alpha_e (v)$ $\forall v\in V$.
\end{enumerate}
Clearly, $K(Q,e)$ is an extension of $Q$.
Also, $Q'':=I(Q,e)$ is an extension of $Q$: 
if $\mbf X\in \End(Q'')$, then $\alpha_e X_a = X_b \alpha_e$ because 
$\alpha_e = \alpha''_{e''} \alpha''_{e^{\sharp}}$. 
\begin{lem}\label{KI} Let $Q=(E_0,E_1,\mbf U,\bs\alpha)$ be a quiver
 representation, and let $e\in E_1$. Let 
$Q'=K(Q,e)$ and $Q''=I(Q,e)$. Then the maps 
$\pi^{Q'}_Q : \End(Q') \ra \End(Q)$ 
 and $\pi^{Q''}_Q: \End(Q'') \ra \End(Q)$ are ring isomorphisms.
\end{lem}   
\begin{proof} As above, let $a=\sigma(e)$, $b=\tau(e)$, $Y=\ker(\alpha_e)$, 
$Z=\im(\alpha_e)$. Let 
$\mbf X=(X_i)_{i\in E_0}\in \End(Q)$. Since $\alpha_e X_a = X_b \alpha_e$, the map $X_a$ preserves $Y$ and $X_b$ preserves $Z$. Any element 
$\mbf X'\in \left(\pi_Q^{Q'} \right)^{-1}(\mbf X)$ satisfies 
$\alpha'_{e'} X'_c = X'_c \alpha'_{e'}$. Since $\alpha'_{e'}$ is injective, it follows that $\left(\pi_Q^{Q'}\right)^{-1}(\mbf X)=\{\mbf X'\}$, where
$X'_c = X_a |_Y$. Similarly, $\left(\pi_Q^{Q''}\right)^{-1} (\mbf X)=
\{ \mbf X'' \}$, 
where $X''_d=X_b|_Z$. Hence, $\pi^{Q'}_Q$ and $\pi^{Q''}_Q$ are bijections.
\end{proof}

\section{Reduction to nilpotent endomorphisms}\label{rednil}

In this section we show that, if $\Cal Y$ is a 
class of endomorphisms over $\mF_q$, the problem of counting 
orbits of $\Cal Y$-endomorphisms of a quiver representation may be reduced
 to that of counting orbits of nilpotent endomorphisms of various quiver
representations. (A \emph{$\Cal Y$-endomorphism} is an endomorphism $\mbf X$ 
such that $X_a\in \Cal Y$ for all $a$.) 
We use standard methods related to rational canonical forms. The results of this section are not used elsewhere in the paper, but provide some motivation for our later
focus on nilpotent endomorphisms.

Let $U$ be a vector space over a field $K$. 
Let $L$ be a field containing $K$. Suppose $U'$ is a vector space over $L$ that becomes $U$ if one restricts the scalars to $K$; that is, $U'=U$ as sets and 
multiplication by scalars from $K$ in $U$ is the same as in $U'$. We shall say
that $U'$ is an $L$-\emph{expansion} of $U$ and $U$ is the 
\emph{restriction} of $U'$ 
to $K$. Let $Q=(E_0,E_1,\mbf U,\bs\alpha)$ be a quiver representation over $K$.
Say that a quiver representation $Q'=(E_0,E_1,\mbf U',\bs\alpha)$ over 
$L$ is an $L$-\emph{expansion} of $Q$ 
(and $Q$ is the \emph{restriction} of $Q'$ to $K$) if, for each $a\in E_0$, the space $U'_a$ is an $L$-\emph{expansion} of $U_a$. (Then $\alpha_e$ is $L$-linear for each 
$e\in E_1$.)  

Let $\mbf X\in \End(Q)$. If $a\in E_0$ and $f$ is a monic irreducible polynomial
 over $\mF_q$, let
$$
U_{\mbf X,f,a} = \{ u\in U_a : f(X_a)^k u = 0 \text{ for some } k\in \mN\}.
$$
Then $U_{\mbf X,f,a}=0$ for all but finitely many $f$. For all monic irreducible polynomials $f\in \mF_q [T]$ and all $e\in E_1$, we have
$\alpha_e (U_{\mbf X,f,\sigma(e)}) \le U_{\mbf X,f,\tau(e)}$. 
Thus, $\mbf U_{\mbf X,f}:=(U_{X,f,a})_{a\in E_0}$ induces a subrepresentation 
$Q_{\mbf X,f}$ of $Q$.

\begin{lem}\label{rednil1} 
Let $Q$ be a quiver representation over $\mF_q$. Suppose 
$\mbf X\in\End(Q)$. Then 
$$ 
Q = \bigoplus_{f} Q_{\mbf X,f}
$$
where the sum is over all monic irreducible polynomials $f$ over $\mF_q$. 
Moreover, the isomorphism class of each component $Q_{\mbf X,f}$ is an invariant of the $\Aut(Q)$-orbit
of $\mbf X$.
\end{lem}  

\begin{proof} The first statement follows from the fact that
$U_{a}=\bigoplus_{f} U_{\mbf X,f,a}$ for each $a$. The second statement is 
clear.
\end{proof}

We now consider each quiver representation $Q_{\mbf X,f}$ separately. Fix a monic
irreducible polynomial $f=f(T)$ over $\mF_q$. Let $d$ be the degree of $f$.   
Suppose $\mbf X$ is an endomorphism of a quiver representation $Q$ over $\mF_q$
 such that $Q=Q_{\mbf X,f}$.   
Let $\mF_q [T]_{(f)}$ be the localisation of $\mF_q[T]$ at the ideal
$(f)$; it consists of all fractions $g/h$ such that
 $g,h\in\mF_q [T]$ and $h$ is not divisible by $f$. Then $\mbf X$ induces an
$\mF_q [T]_{(f)}$-module structure on each $U_a$: 
multiplication by $T$ is given
by the action of $X_a$. Moreover, each $\alpha_e$ becomes 
an $\mF_q [T]_{(f)}$-module homomorphism.
 
Let $S_{f}$ be the completion of the discrete valuation ring $\mF_q [T]_{(f)}$. 
That is, $S_{f}$ is the inverse limit of the rings 
$\mF_q [T]/(f^k)$, $k=1,2,\ldots$. 
Any finite $\mF_q [T]_{(f)}$-module is annihilated by $f^k$ 
for large enough $k$.
Thus, a finite $\mF_q [T]_{(f)}$-module may be 
thought of as a (finite) $S_{f}$-module (and vice versa). 

Now, by~\cite[\S 9, Proposition 3]{Bourbaki}, $S_f$ is isomorphic to 
$\mF_{q^d}[[Z]]$, the ring of formal power series over a variable $Z$. Indeed, let
$r$ be the element of $S_f$ whose projection onto $\mF_q [T]/(f^k)$ is  
$T^{q^{kd}}$ for each $k\in \mN$. 
Then $\mF_q [r]\subseteq S_f$ is a field isomorphic 
to $\mF_{q^d}$, and each element of $S_f$ may be represented as a power series
in $f$ with coefficients in $\mF_q[r]$ (see \emph{loc. cit.} for more detail). 

Hence, $Q$ gives rise to a finite $\mF_{q^d} [[Z]]$-module. 
This module corresponds to an $\mF_{q^d}$-expansion $Q'$ of $Q$ and a nilpotent endomorphism of $Q'$ (given by multiplication by $Z$). 

Conversely, an $\mF_{q^d}$-expansion $Q'$ of
$Q$ together with a nilpotent endomorphism of $Q'$ gives rise to an 
$\mF_{q^d}[[Z]]$-module structure on each $U_a$ in such a way that all
$\alpha_e$ are $\mF_{q^d}[[Z]]$-endomorphisms. Identifying $\mF_{q^d}[[Z]]$ with
$S_f$ as above, we get an endomorphism $\mbf X$ of $Q$ such that 
$Q_{\mbf X,f}=Q$. We have proved the following result.

\begin{lem}\label{rednil2} Let $q$ be a prime power, and let $f$ be a monic 
irreducible polynomial over $\mF_q$ of degree $d$. Suppose $Q$ is a quiver 
representation over $\mF_q$. Then the $\Aut(Q)$-orbits of endomorphisms
$\mbf X$ of $Q$ such that $Q_{\mbf X,f}=Q$ are in a one-to-one correspondence
with $\Aut(Q)$-orbits of pairs $(Q',\mbf Z)$ such that
$Q'$ is an $\mF_{q^d}$-expansion of $Q$ and $\mbf Z\in N(Q')$. 
\end{lem}

Let $\Cal Y$ be a class of linear endomorphisms over $\mF_q$, as defined in 
Section~\ref{parintro}. Let 
$$
\End_{\Cal Y}(Q) = \{ \mbf X\in \End(Q): X_a\in \Cal Y \text{ for all }
a\in E_0  \}. 
$$  
Let $F$ be a subset of $\ol{\mF}_q$ such that $\Cal Y=\Cal Y_F$. 
Call a polynomial 
$f\in \mF_q[T]$ a \emph{$\Cal Y$-polynomial} if all the roots of $f$ 
(over $\ol{\mF}_q$) are in $F$. Let $k_{\Cal Y,d}(q)$ be the number of monic
irreducible $\Cal Y$-polynomials of degree $d$.     

Consider again an arbitrary quiver representation 
$Q=(E_0,E_1,\mbf U,\bs\alpha)$ over $\mF_q$.
Lemmata~\ref{rednil1} and~\ref{rednil2} allow us to express
 $\gamma(\Aut(Q),\End_{\Cal Y}(Q))$ in terms of 
$\theta(Q'_0)$ where $Q_0$ varies among direct summands of $Q$ 
and $Q'_0$ is an expansion of $Q_0$.

Let $\mbf X\in \End_{\Cal Y}(Q)$. By Lemma~\ref{rednil1}, there exists
 a finite set $\{f_i\}_{i\in J}$ of irreducible $\Cal Y$-polynomials such that
$Q=\bigoplus_{j\in J} Q_j$ where $Q_j=Q_{\mbf X,f_j}$. Here, $J$ is a finite indexing set. Let $\epsilon(j)=\deg f_j$.
By Lemma~\ref{rednil2}, each $Q_j$ together with the restriction of 
$\mbf X$ to $Q_j$ corresponds to an $\mF_{q^{\epsilon(j)}}$-expansion 
$Q'_j$ of $Q_j$ 
together with a nilpotent endomorphism $\mbf Z^{(j)}\in N(Q'_{j})$. 
Up to conjugation by $\Aut(Q'_j)$, the endomorphism
$\mbf Z^{(j)}$ may be chosen in $\theta(Q'_j)$ ways (by definition). Let 
$\Cal A$ be the partition of $J$ that consists of the sets
$$
\{ j\in J: Q'_j\simeq Q'_i \} 
$$
where $i$ varies among the elements of $J$. In particular, 
$\epsilon(i)=\epsilon(j)$ whenever $i$ and $j$ are $\Cal A$-equivalent.   

\begin{prop}\label{rednil3} 
Let $Q$ be a quiver representation over $\mF_q$. Then
\begin{eqnarray*}
\gamma(\Aut(Q),\End_{\Cal Y}(Q)) & = & 
\sum_{(Q_i)_{i\in J}} \sum_{\Cal A}
 \frac{1}{\prod_{A\in \Cal A} |A|!}
\sum_{\epsilon} 
\prod_{d=1}^{\infty}
\frac{k_{\Cal Y,d}(q)!}{(k_{\Cal Y,d}(q)-|\epsilon^{-1}(d)|)!} \\
& \times & \sum_{(Q'_A)_{A\in \Cal A}} \prod_{A\in\Cal A} \theta(Q'_A)^{|A|}.
\end{eqnarray*}
Here, the first sum is over all decompositions $Q=\bigoplus_{i\in J} Q_i$ 
of $Q$ as a direct sum of representations; two such decompositions are 
considered to be the same if one may be obtained from the other by permuting the indices $i$ and replacing each $Q_i$ with an isomorphic representation 
(we assume that $J=[1,|J|]$). 
The second sum is over all 
partitions $\Cal A$ of $J$ such that $Q_i\simeq Q_j$ whenever $i$ and $j$ are
$\Cal A$-equivalent. 
The third sum is over all maps 
$\epsilon: J\ra \mN$ such that $\epsilon(i)=\epsilon(j)$ whenever $i$ and 
$j$ are $\Cal A$-equivalent. 
The last sum is over all isomorphism classes of 
tuples $(Q'_A)_{A\in \Cal A}$
 where $Q'_A$ is an $\mF_{q^{\epsilon(i)}}$-expansion of 
$Q_i$ ($i$ is an arbitrary element of $A$) and the quiver representations 
$Q'_{A}$ ($A\in \Cal A$) are pairwise not isomorphic.  
\end{prop} 
\begin{proof} Suppose $\Cal A$, $\epsilon$, $Q'_A=Q'_i$ ($i\in A$), as above,
 are fixed. There are 
$$
\prod_{d=1}^{\infty}\frac{k_{\Cal Y,d}(q)!}{(k_{\Cal Y,d}(q)-|\epsilon^{-1}(d)|)!}
$$  
ways to assign a monic irreducible $\Cal Y$-polynomial $f_i$ of 
degree $\epsilon(i)$ to each $i\in J$ so that
 all those polynomials are distinct. (Note that all but finitely many
 terms in the product are equal to $1$, so the product is well defined.) 
There are 
$$
 \prod_{A\in \Cal A} \theta(Q'_A)^{|A|}
$$
ways to choose, for each $i\in J$, 
an $\Aut(Q'_i)$-orbit in $N(Q'_i)$.
By Lemmata~\ref{rednil1} and~\ref{rednil2}, 
these assignments determine an $\Aut(Q)$-orbit in $\End_{\Cal Y}(Q)$, and all
orbits occur in this way. 
However, a permutation of the indices within a subset $A\in \Cal A$ 
yields the same orbit. Thus, each orbit giving rise to this particular $\Cal A$ is obtained by   
$
\prod_{A\in \Cal A} (|A|!)
$
such assignments. Hence, we must divide by this number to obtain the number 
of orbits.
\end{proof}

Now let $\mbf l=(l_1,\ldots,l_s)$ be a tuple of nonnegative integers, and consider the quiver $R_{\mbf l}(q)$.
If $\mbf l'=(l'_1,\ldots,l'_s)$ is another such $s$-tuple, let 
$\mbf l+\mbf l'=(l_1+l'_1,\ldots,l_s+l'_s)$; write $\mbf l\le \mbf l'$ if 
$l_i\le l'_i$ for all $i$. Also, if $b\in \mQ$, let 
$b\mbf l = (b l_1,\ldots, b l_s)$. It is easy to see that all decompositions
 of $R_{\mbf l}$ into direct sums of subrepresentations are of the form
$$
R_{\mbf l} = \bigoplus_{i=1}^n R_{\mbf l^i}
$$ 
where $\mbf l = \mbf l^1+\cdots +\mbf l^n$. 
Let $d\in \mN$. If not all $l_i$ 
are divisible by $d$, then there is no $\mF_{q^d}$-expansion of $R_{\mbf l}(q)$.
If all $l_i$ are divisible by $d$, then $R_{\mbf l}(q)$ has a unique
(up to the action of $\Aut(R_{\mbf l}(q))$) $\mF_{q^d}$-expansion, 
namely $R_{\mbf l/d}(q^d)$. 

Consider the set of tuples 
$J=(\mbf l^1,\ldots,\mbf l^n,d_1,\ldots,d_n)$ such that 
each $\mbf l^i$ is an $s$-tuple of nonnegative integers, not all equal to zero,
 $d_i\in \mN$ for each $i$  and 
$$
\sum_{i=0}^n d_i \mbf l^i = \mbf l.
$$ 
Let $\scr D(\mbf l)$ be a complete set of representatives of the action of 
the symmetric group $S_n$ on this set, where the action is by permuting the 
indices $1,\ldots,n$. Less formally, $\scr D(\mbf l)$ is in a one-to-one 
correspondence with ways to decompose $R_{\mbf l}(q)$ as a direct sum of 
quivers $R_{\mbf l'}(q)$ and to expand each 
$R_{\mbf l'}(q)$ to $R_{\mbf l'/d}(q^d)$ for some $d$.  

By Lemma~\ref{quivpar}, the numbers $\gamma(P^{\mbf l}(q),P^{\mbf l}(q))$ and 
$\gamma(P^{\mbf l}(q),M^{\mbf l}(q))$ may both be expressed as 
$\gamma(\Aut(R_{\mbf l}(q)),\End_{\Cal Y}(R_{\mbf l}(q)))$ where $\Cal Y$ is 
the class of all invertible endomorphisms or the class of all endomorphisms, 
respectively. If $\Cal Y$ is one of those classes, 
then $k_{\Cal Y,d}(q)$ is polynomial in $q$. We deduce the following from
Proposition~\ref{rednil3}.  

\begin{cor}\label{rednil4} 
Let $\mbf l=(l_1,\ldots,l_s)$ be a tuple of nonnegative integers. Then there
exist tuples $(a_J)_{J\in \scr D(\mbf l)}$ and 
$(b_{J})_{J\in \scr D(\mbf l)}$ 
where $a_{J}(T)$ and $b_{J}(T)$ are 
polynomials with rational coefficients such that, for all prime powers $q$, 
\begin{eqnarray*}
\gamma(P^{\mbf l}(q),P^{\mbf l}(q))& = & 
\sum_{J\in \scr D(\mbf l)} a_J (q) \prod_{i=1}^n \rho_{\mbf l^i}(q^{d_i})
\quad
\text{and } \\
\gamma(P^{\mbf l}(q),M^{\mbf l}(q)) & = & 
\sum_{J\in \scr D(\mbf l)} b_J (q) \prod_{i=1}^n \rho_{\mbf l^i}(q^{d_i})
\end{eqnarray*}
In each case, the first sum is over all elements 
$J=(\mbf l^1,\ldots,\mbf l^n,d_1,\ldots,d_n)$ of $\scr D(\mbf l)$.
\end{cor}

\section{Preliminary results}\label{parprem}

In this section we state several 
standard and straightforward results. We prove the last of these results; 
the first two are easy exercises. 

\begin{lem}\label{normal2} Let $G$ be a group acting on a set $Y$. Let $N$ be a normal subgroup of $G$, and let $S$ be the set of $N$-orbits on $Y$. The action of $G$ on $Y$ induces an action of $G/N$ on $S$ via $gN\circ Ny = Ngy$ 
(for all $g\in G$, $y\in Y$). Moreover, the $G$-orbits on $Y$ are 
in a one-to-one correspondence with the $G/N$-orbits on $S$.
\end{lem}

\begin{lem}\label{actprod} Let $G$ be a group acting on finite sets 
$X$ and $Y$. Suppose $S$ is a subset of $X\times Y$ preserved by $G$. 
Let $R$ be a complete set of representatives of $G$-orbits on $X$ (so each $G$-orbit on $X$ contains exactly one element of $R$). Then 
$$\gamma(G,S)=\sum_{x\in R} \gamma(\Stab_G (x), \{y\in Y: (x,y)\in S \}).$$
\end{lem}

Recall that if $V$ and $W$ are vector spaces and 
$X\in \End(V)$, $Y\in \End(W)$, then 
$$
\Cal I(X,Y)=\{ T\in \Hom(V,W): TX=YT \}.
$$

\begin{lem}\label{dimI} Let $V$ and $W$ be finite dimensional vector spaces over a 
field $K$. Suppose $X\in \End(V)$ and $Y\in \End(W)$. Then 
$$
\dim \Cal I(X,Y) = \dim \Cal I(Y,X).
$$ 
\end{lem}

\begin{proof} Let $V^*$ and $W^*$ be the dual spaces to $V$ and $W$. 
Let $X^*\in \End(V^*)$ and $Y^*\in \End(W^*)$ be the maps dual to $X$ and $Y$:
$(X^* f)(v)= f(X(v))$ for all $f\in V^*$, $v\in V$. Then 
$\dim \Cal I(Y,X)= \dim\Cal I(X^*,Y^*)$: a linear bijection between 
$\Cal I(Y,X)$ 
and $\Cal I(X^*,Y^*)$ is given by assigning to each map its dual.  
Also, $X^*$ and $X$ have the same rational canonical form, as do
 $Y$ and $Y^*$. Thus,
$$
\begin{array}{cr}
\quad\qquad\qquad \dim\Cal I(X,Y)= \dim\Cal I(X^*,Y^*)= \dim\Cal I(Y,X). 
\qquad\qquad\quad  &  \qed
\end{array} 
$$
\renewcommand{\qed}{}\end{proof}

\section{Action on the dual space}\label{dualact}

Let $V$ be a finite dimensional vector space over a finite field $\mF_q$.  
If $w\in V^*$ and $v\in V$, 
we shall write $(w,v)$ for $w(v)$. 
Suppose a group $G$ acts on $V$ by linear maps, 
so a representation of $G$ on $V$ is given: 
$$(g,v) \mapsto gv.$$ 
This action gives rise to a representation of $G$ on $V^*$ in the usual way: 
$$((gw,v)=(w,g^{-1}v),$$
where $g\in G$, $w\in V^*$, $v\in V$. We shall rely on the following simple 
result, which is a weak version of Brauer's Theorem (\cite[11.9]{CRI}). 

\begin{lem} \label{duality} The number of $G$-orbits on $V$ is equal to the number of $G$-orbits on $V^*$.
\end{lem}
\begin{proof} By a well-known orbit-counting formula, 
$$\gamma(G,V)=\frac{1}{|G|}\sum_{g\in G} \Fix_V (g),$$
where $\Fix_V (g)$ is the number of elements of $V$ fixed by $g$. 
The same holds for $V^*$. 
Therefore, it suffices to show that $\Fix_V (g)=\Fix_{V^*} (g)$ 
for all $g\in G$. 

Let $g\in G$. Fix a basis of $V$, and let $A$ be the matrix of the action of $g$ with respect to this basis. Then $(A^t)^{-1}$ is the matrix of the action of 
$g$ on $V^*$ with respect to the dual basis. Since 
$\rank(A-I)=\rank(A^t-I)=\rank((A^t)^{-1}-I)$, 
$$
\qquad\qquad \Fix_V (g)=q^{n-\rank(A-I)}=q^{n-\rank((A^t)^{-1}-I)}=\Fix_{V^*} (g). 
\qquad \qquad \qed
$$
\renewcommand{\qed}{}\end{proof}

Let $Q=(E_0,E_1,\mbf U ,\bs\alpha)$ be a quiver representation. 
Let $a,b\in E_0$. 
Define an extension 
$\Omega(Q,a,b)=Q'=(E'_0, E'_1,\mbf U',\bs\alpha')$ of $Q$ as follows:
\begin{enumerate}[(i)]
\item $E'_0 = E_0 \sqcup \{ c \}$; 
\item $\dim U'_c = \dim U_a + \dim U_b$;
\item $E'_1=E_1\sqcup \{e_1,e_2\}$ where $\sigma(e_1)=a$, $\tau(e_1)=c=\sigma(e_2)$,
 $\tau(e_2)=b$; 
\item $\alpha'_e=\alpha_e$ for all $e\in E_1$;
\item $\alpha_{e_1}$ is injective, $\alpha_{e_2}$ is surjective, and 
$\im(\alpha_{e_1})=\ker(\alpha_{e_2})$, 
so the sequence
$$
0\lra U_a\stackrel{\alpha_{e_1}}{\lra} U'_c \stackrel{\alpha_{e_2}}{\lra} 
U_b \lra 0
$$
is exact.
\end{enumerate}  
Write $\pi$ for $\pi^{Q'}_Q:\End (Q')\ra \End(Q)$. Note that $\Aut(Q)$ acts 
on $\Hom(X_a,X_b)$ in a natural way: $\mbf g\circ R=g_b R g_a^{-1}$ 
($\mbf g\in \Aut(Q)$, $R\in \Hom(X_a,X_b)$). 
\begin{prop}\label{link}  If $\mbf X$ is an endomorphism of $Q$ and $H$ is
a subgroup of $\Aut(Q)$ that fixes $\mbf X$, then
$$\gamma(H^{Q'},\pi^{-1}(\mbf X)) = \gamma(H,\Cal{I}(X_a,X_b)).$$
\end{prop}
\begin{proof} Let $k=\dim U_a$, $m=\dim U_b$.  
Let $\beta=\alpha_{e_1}$.
Let $V=\beta(U_a)\le U'_c$. 
Choose a complement $W$ of $V$ in $U'_c$, so that 
$U'_c=V\oplus W$. Then
$\epsilon:=\alpha_{e_2}|_{W}$ is a bijective linear map from $W$ onto $U_b$. 

For any $R\in\Hom(U_b,U_a)$, let $\phi (R)$ be the unique element of
$\pi^{-1}(\mbf X)$ such that, for $v\in V$, $w\in W$,
$$
\phi(R)_c (v+w)= \beta X_a \beta^{-1} \cdot v + 
\epsilon^{-1} X_b \epsilon \cdot w 
+ \beta R \epsilon \cdot w.$$ 
Less formally, 
if one chooses bases for $U_a$ and $U_b$, 
uses $\beta$ and $\epsilon^{-1}$ to get bases of $V$ and $W$,
combines these to get a basis
 of $U'_c$ and represents linear operators as matrices using these bases, then
$$
\phi(R)_c=\begin{pmatrix} 
X_a & R \\
0 & X_b 
\end{pmatrix}.
$$ 
Clearly, $\phi$ is a bijection from $\Hom(U_b,U_a)$ onto $\pi^{-1}(\mbf X)$.
For $T\in\Hom(U_b,U_a)$, define 
$\psi(T)\in H^{Q'}$ as follows: $\psi(T)_d=I_{U_d}$ for $d\ne c$, and 
$$
\psi(T)_c (v+w) = v + w + \beta T \epsilon \cdot w 
\quad \forall v\in V, w\in W.
$$
As a matrix, 
$$
\psi(T)_c = \begin{pmatrix} I_k & T \\
0 & I_m 
\end{pmatrix}.
$$  
Consider the surjective group homomorphism 
$\pi':=\pi|_{H^{Q'}}: H^{Q'}\ra H$. Let $A=\ker\pi'$. 
Then $\psi:\Hom(U_b,U_a)\ra A$ is an isomorphism of abelian groups. 
A direct calculation (e.g. using matrices) 
produces a formula for the action of $A$ on 
$\pi^{-1}(\mbf X)$:
$$
\psi(T) \phi(R) \psi(T)^{-1} = \phi(R + T X_b - X_a T).
$$
Let 
$$D=\{T X_b - X_a T: \: T\in \Hom(U_b,U_a) \}.$$
Then the map $\phi^{-1}: \pi^{-1}(\mbf X)\ra \Hom(U_b,U_a)$ establishes a one-to-one correspondence between $A$-orbits on $\pi^{-1}(\mbf X)$ 
and cosets of $D$ in $\Hom(U_b,U_a)$. 
Hence, the maps $\pi'$ and $\phi$ establish an isomorphism between the action
of $H^{Q'}/A$ on the set of $A$-orbits in $\pi^{-1}(\mbf X)$ and the natural 
action of $H$ on $\Hom(U_b,U_a)/D$.
It follows, by Lemma \ref{normal2} (applied to the normal subgroup $A$ of $H^{Q'}$), that
$$\gamma (H^{Q'},\pi^{-1} (\mbf X)) = \gamma(H, \Hom(U_b,U_a)/D).$$ 

The dual space to $\Hom(U_b,U_a)$ may be identified with $\Hom(U_a,U_b)$, where
$$
(S,R)=\tr(SR) \qquad \forall R\in \Hom(U_b,U_a), \; S\in \Hom(U_a,U_b),
$$
and the dual action of $H$ on the space $\Hom(U_a,U_b)$ is induced by 
the natural action of 
$\GL(U_a)\times \GL(U_b)$ on this space.
Hence the dual space of $\Hom(U_b,U_a)/D$ is 
$$D^{\perp}=\{S\in \Hom(U_a,U_b): \tr(RS)=0 \text{ for all } R\in D \}.$$

Any $S\in \Cal{I}(X_a,X_b)$ belongs to $D^{\perp}$. Indeed, for all 
$T\in\Hom(U_b,U_a)$, 
\begin{eqnarray*}
\tr(S(T X_b - X_a T))& = & \tr(S T X_b)-\tr(S X_a T)=\tr(X_b ST)-\tr(S X_a T) \\
& = & \tr((X_b S - S X_a)T)=0.
\end{eqnarray*}
Moreover,
$$
\dim D^{\perp}=km-\dim D=\dim \Cal I(X_b,X_a) = \dim \Cal I(X_a,X_b).
$$
(The last equality holds by Lemma~\ref{dimI}.) It follows that
$D^{\perp}=\Cal I(X_a,X_b)$.
By Lemma~\ref{duality}, 
$$\gamma(H,\Hom(U_b,U_a)/D)=\gamma(H,D^{\perp})=\gamma(H,\Cal{I}(X_a,X_b)),
$$
and the result follows.
\end{proof}

Proposition~\ref{link} will allow us to replace the quiver $Q'=\Omega(Q,a,b)$ 
with quivers obtained from $Q$ by adding an arrow from $U_a$ to $U_b$ when counting $\Aut(Q')$-orbits. A more usable form of this proposition is given by the corollary below.  

Let $Q=(E_0,E_1,\mbf U,\bs\alpha)$, $a,b\in E_0$, $Q'=\Omega(Q,a,b)$, 
$\pi=\pi^{Q'}_Q$ be as above. Let $G$ be a subgroup of $\Aut(Q)$. 
Let $\Xi=\Xi(Q,G,a,b)$ be a complete set of representatives of 
$G$-orbits on $\Hom(U_a,U_b)$ (the choice of $\Xi$ does not matter).  
For each $S\in\Xi$, 
let $Q^S$ be the quiver 
$(E_0, E_1 \sqcup \{ f \}, \mbf U, \alpha^S)$  where 
$\sigma(f)=a$, $\tau(f)=b$, $\alpha^S_e=\alpha_e$ for $e\in E_1$, 
and $\alpha^S_f=S$. That is, $Q^S$ is obtained from $Q$ by adding the 
linear map $S:U_a\ra U_b$.   
\begin{cor}\label{genlink} 
Let $B$ be a subset of $\End(Q)$ preserved by a subgroup $G$ of $\Aut(Q)$. Then
$$
\gamma(G^{Q'},\pi^{-1} (B)) = 
\sum_{S\in\Xi} \gamma (G\cap \Aut(Q^S), B\cap \End(Q^S)). 
$$ 
\end{cor}
\begin{proof} Let $Y$ be a complete set of representatives of $G$-orbits on 
$B$.
Let 
$$
Z=\{ (S,\mbf X)\in \Hom(U_a,U_b)\times B : S X_a = X_b S \}.
$$ 
Then 
\begin{eqnarray*}
\gamma(G^{Q'},\pi^{-1} (B) ) & = & 
\sum_{\mbf X\in Y} \gamma (\Stab_G (\mbf X)^{Q'}, \pi^{-1}(\mbf X)) \\
 & = & \sum_{\mbf X\in Y} \gamma (\Stab_G (\mbf X), \Cal{I} (X_a,X_b)) \\
 & = & \gamma (G, Z) = 
\sum_{S\in\Xi} \gamma (G\cap \Aut(Q^S), B\cap \End(Q^S)). 
\end{eqnarray*}
The first equality holds by Lemma \ref{actprod}, applied to the set
$$\{ (\mbf X, \mbf X')\in B\times\End(Q') : \pi(\mbf X')=\mbf X \}.$$
The second one follows from Proposition \ref{link}. 
The third and fourth equalities follow from Lemma \ref{actprod}, 
applied to the set $Z$ in two different ways. 
\end{proof}

The rest of this section is devoted to an informal sketch of a proof of  Theorem~\ref{pargen} in the 
special case when $\Cal Y$ is the class of nilpotent matrices. 
A formal proof of a more general statement is 
in the next section. Let $Q$ be the quiver
representation consisting of vector spaces $U_a=\mF_q^k$ and $U_b=\mF_q^m$ without any arrows. Let
$Q'=\Omega(Q,a,b)$, so $Q'$ is the quiver representation
$$
\xymatrix{
\mF_q^k\ar@{^{(}->}[r] & \mF_q^{k+m}\ar@{->>}[r] & \mF_q^m 
}
$$
where the first map is injective, the second one is surjective, and the sequence is 
exact at $\mF_q^{k+m}$.
It is not difficult to see that $\rho_{(k,m)}(q)=\theta(Q')$. The orbit of an element 
$S_1\in \Hom(\mF_q^k,\mF_q^m)$
with respect to the action of $\Aut(Q)=\GL_k (q)\times \GL_m (q)$ is determined by $\rank (S_1)$. 
Thus, by Corollary~\ref{genlink} (applied to $G=\Aut(Q)$ and $B=N(Q)$), $\rho_{(k,m)}(q)=\theta(Q')$ may be expressed as a sum, with one summand for each possible value of $\rank(S_1)$. If $S_1=0$, we get the summand $\theta(Q)=p(k)p(m)$.
Otherwise, the summand corresponding to $S_1$ is $\theta(Q_1)$ where $Q_1$ is the quiver representation
$$
\mF_q^k \stackrel{S_1}{\lra} \mF_q^m. 
$$
Let $\lda_1=\rank S_1$. Then $\ker S_1$ and $\im S_1$ may be identified with $\mF_q^{k-\lda_1}$ and 
$\mF_q^{\lda_1}$ respectively.
It follows from Lemma~\ref{KI} that $\theta(Q_1)=\theta(\ol{Q}_1)$ where 
$\ol{Q}_1$ is the quiver representation
$$
\xymatrix{
\mF_q^{k-\lda_1} \ar@{^{(}->}[r] & \mF_q^k \ar@{->>}[r]^{S_1} &
\mF_q^{\lda_1}\ar@{^{(}->}[r] & \mF_q^m  
}
$$
with the sequence exact at $\mF_q^k$. We apply Corollary~\ref{genlink} again:
$\theta(\ol{Q}_1)$ may be expressed as a sum of terms corresponding to maps 
$S_2\in \Hom(\mF_q^{k-\lda_1},\mF_q^{\lda_1})$ of varying ranks. (Again, we have one map of each 
possible rank.) Thus, we branch into cases corresponding to the possible values of $\lda_2:=\rank S_2$. 

We continue this argument inductively, considering maps 
$S_3,S_4,\ldots$ of ranks $\lda_i:=\rank S_i$ 
until $S_{s+1}=0$ for some $s$. 
(Here, $S_{i+1} : \mF_q^{k-\lda_1-\cdots-\lda_i} \ra \mF_q^{\lda_i}$.) Ultimately, we express 
$\theta(Q')=\rho_{(k,m)}(q)$ as a sum of terms indexed by partitions $\lda=(\lda_1,\ldots,\lda_s)$. 
The term corresponding to $\lda$ is equal to $\theta(Q_s)$ where $Q_s$ is the quiver representation
$$
\xymatrix{
\mF_q^{k-|\lda|} & \mF_q^{\lda_s} \ar@{^{(}->}[r] & \mF_q^{\lda_{s-1}}\ar@{^{(}->}[r] 
& \cdots \ar@{^{(}->}[r] & \mF_q^{\lda_1} \ar@{^{(}->}[r] & \mF_q^m. 
}
$$  
It follows from Lemma~\ref{quivpar} that $\theta(Q_s)=p(k-|\lda|) \nu_{\lda}^m (q)$. Putting $j=|\lda|$,
we obtain
$$
\rho_{(k,m)}(q)= \sum_{j=0}^k p(k-j) \sum_{\lda\in \Cal P^m_j} \nu_{\lda}^m (q),
$$ 
as stated in Theorem~\ref{pargen}.

\section{An inductive argument}\label{indarg} 

Fix a positive integer $m$. Let $W$ be an $m$-dimensional vector space over $\mF_q$. 
Let $G$ be a subgroup of $\GL(W)$, and let $A\subseteq \End(W)$ be a 
set preserved by $G$. 

If $\lambda=(\lambda_1,\ldots,\lambda_s)\in \Cal P^m$, 
let $\Cal F_{\lambda}$ be the set of flags 
$\mbf W = (W_1,W_2,\ldots,W_s)$ such that
$$
W\ge W_1 \ge W_2 \ge \cdots \ge W_s > 0
$$
and $\dim W_i = \lda_i$.
If $\mbf W=(W_1,\ldots,W_s)\in \Cal F_{\lda}$ and
$\mbf W'=(W_1,\ldots,W_{s+1}) \in \Cal F_{\lda'}$ for an appropriate 
$\lda'=(\lda_1,\ldots,\lda_{s+1})$ (that is, the first $s$ subspaces are the same for $\mbf W$ and $\mbf W'$), we shall write $\mbf W' \succ \mbf W$.  
For each $\lambda\in \Cal P^m$, choose a complete set $F_{\lambda}$ 
of representatives of $G$-orbits on $\Cal F_{\lambda}$. 
We may assume that these choices satisfy the following property: 
if $\mbf W\in \Cal F_{\lda}$, $\mbf W' \in  F_{\lda'}$ and 
$\mbf W'\succ \mbf W$, then $\mbf W \in F_{\lda'}$. 
Let $\lda\in \Cal P^m$. Let
$$
T_{\lda}(A)= \{ (\mbf W,R)\in \Cal F_{\lambda}\times A : \;
R(W_i)\subseteq W_i \text{ for all } i \}, \text{ and let }
$$
$$\xi_{\lambda}(G,A)=\gamma(G,T_{\lda}(A)).$$   
Let $\lda=(\lda_1,\ldots,\lda_s)\in \Cal P^m$, and 
let $\mbf W=(W_1,\ldots,W_s)\in\Cal F_{\lda}$. Let $Q_{\lda}(\mbf W)$ 
be the quiver representation 
$$
\xymatrix{
W_s\ar@{^{(}->}[r] &  W_{s-1}\ar@{^{(}->}[r] &  \cdots\ar@{^{(}->}[r] & 
 W_1\ar@{^{(}->}[r] &  W
}
$$
where the arrows are the inclusion maps (obtained by restricting $I_W$).
If $s=0$ (so $\lda=()$), we get the quiver representation 
$Q_{()}$ consisting just of the space
 $W$ without any arrows. 

Let $\Cal Y$ be a class of linear operators over $\mF_q$. 
We shall assume that $A\subseteq \Cal Y$.
Let 
$Q$ be an extension of $Q_{()}$. 
Let 
$\pi=\pi^{Q}_{Q_{()}}: \End(Q)\ra \End(Q_{()})$ be the natural projection map. 
Recall that 
\begin{eqnarray*}
G^Q & = & \pi^{-1} (G) \cap \Aut(Q), \qquad \text{and let} \\
A^Q_{\Cal Y} & = & \pi^{-1} (A)\cap \End_{\Cal Y}(Q).
\end{eqnarray*}

By Lemma~\ref{quivpar}, for any $\mbf W\in \Cal F_{\lda}$, 
$\Stab_G (\mbf W)$-orbits on 
$\{R\in A: R(W_i)\subseteq W_i \;\forall i\}$
are in a one-to-one correspondence with $G^{Q_{\lda}(\mbf W)}$-orbits on 
$A^{Q_{\lda}(\mbf W)}$.
It follows, by Lemma~\ref{actprod}, that
\begin{equation}\label{xi}
\xi_{\lda} (G,A)= 
\sum_{\mbf W\in F_{\lda}} \gamma(G^{Q_{\lda}(\mbf W)},A^{Q_{\lda}(\mbf W)}).
\end{equation}

Let $\mbf W\in \Cal F_{\lda}$ and $k\in \mZ_{\ge 0}$. 
Let $D^{k}_{\lda}(\mbf W)$ be the quiver representation
$$
\xymatrix{
V & & W_s\ar@{^{(}->}[r] &  W_{s-1} \ar@{^{(}->}[r] & \cdots 
\ar@{^{(}->}[r] &  W_1 \ar@{^{(}->}[r] & W,
}
$$
where $\dim V=k$ and the maps are as in $Q_{\lda}(\mbf W)$ (i.e. injective).
Let $Q_{\lda}^k (\mbf W)$ be the quiver representation 
$\Omega (D_{\lda}^k (\mbf W), a, b)$ 
(as defined in the previous section) where $a$ and $b$ correspond to $V$ and
 $W_s$ respectively. Then $Q_{\lda}^k (\mbf W)$ 
may be depicted as 
$$
\xymatrix{
V \ar@{^{(}->}[r] & Z \ar@{->>}[r] &  
W_s \ar@{^{(}->}[r] & W_{s-1} \ar@{^{(}->}[r] &  \cdots 
\ar@{^{(}->}[r] &  W_1 \ar@{^{(}->}[r] & W
}
$$ 
(Note that $\dim Z=k+\lda_s$ and the sequence is exact at $Z$.)
Consider again the case $\lda=()$. (Then $\mbf W$ is the empty sequence and is omitted
 from the notation.) 
The corresponding quiver representation 
$D_{()}^k$ does not have any arrows and consists 
just of the vector spaces $V$ ($\dim V=k$) and $W$. 
The quiver representation $Q_{()}^k$ is 
\begin{equation}\label{Qk}
\xymatrix{
V \ar@{^{(}->}[r] & Z \ar@{->>}[r] &  W,
}  
\end{equation}
where the sequence is exact at $Z$.
Write $D^k$ for $D_{()}^k$ and $Q^k$ for $Q_{()}^k$. Let $G^{(k)}=G^{Q^k}$ and 
$A^{(k)}_{\Cal Y}=A^{Q^k}_{\Cal Y}$. Recall that $c(k,\Cal Y)$ is the number
of $\GL(V)$-orbits on $\End_{\Cal Y} (V)$. We shall
prove the following result using Corollary~\ref{genlink}.
\begin{lem}\label{indlem}
 Let $\lda=(\lda_1,\ldots,\lda_s)\in \Cal P^m$ and $k\in \mZ_{\ge 0}$. 
Let $\mbf W \in F_{\lda}$. Then  
\begin{eqnarray*}
\gamma\left(G^{Q_{\lda}^k (\mbf W)}_{\vph Y}, A^{Q_{\lda}^k (\mbf W)}_{\Cal Y}\right)  & = & 
c(k,\Cal Y)
\gamma\!\left(G^{Q_{\lda}(\mbf W)}_{\vph Y}, A^{Q_{\lda}(\mbf W)}_{\Cal Y} \right) \\
& + &  \sum_{\lda_{s+1}=1}^{\min(\lda_s,k)} 
\sum_{\substack{\mbf W'\in F_{\lda'} \\ \mbf W' \succ \mbf W}}
\gamma\left(G^{Q_{\lda}^{k-\lda_{s+1}}(\mbf W')}_{\vph Y}, 
A_{\Cal Y}^{Q_{\lda}^{k-\lda_{s+1}} (\mbf W')} \right),
\end{eqnarray*}
where $\lda'=(\lda_1,\ldots,\lda_s,\lda_{s+1})$.
\end{lem}
\begin{proof} 
Let $\Phi$ be the set of subspaces $W_{s+1}$ of $W_s$ such that 
$$
\mbf W':=(W_1,\ldots,W_s,W_{s+1})\in F_{(\lda_1,\ldots,\lda_s,\lda_{s+1})}
$$
for some $\lda_{s+1}\in [1,k]$ (of course, then $\dim W_{s+1}=\lda_{s+1}$). 
Then $\Phi$ is a complete set of representatives of $\Stab_G (\mbf W)$ 
on non-zero subspaces of $W_s$ of dimension at most $k$. 
For each $W_{s+1}\in \Phi$, choose
a linear map $\alpha_{W_{s+1}}:V \ra W_s$ with image $W_{s+1}$.

Let $H=G^{D_{\lda}^k (\mbf W)}$. Then $H=\GL(V)\times G^{Q_{\lda}(\mbf W)}$. Thus,
any two elements of $\Hom(V,W_s)$ are 
$H$-conjugate if and only if their images in $W_s$ are 
$\Stab_G (\mbf W)$-conjugate. Let $\Xi$ be the set consisting of the zero map 
$V\ra W_s$ and all the maps $\alpha_{W_{s+1}}$ where $W_{s+1}$ runs through $\Phi$. 
Then $\Xi$ is a complete set of representatives of $H$-orbits on $\Hom(V,W_s)$. 
 For each $\alpha\in\Xi$, let $L^{\alpha}$ be the quiver
obtained from $D_{\lda}^k (\mbf W)$ by adding the map $\alpha:V\ra W_s$, that is,
$$
\xymatrix{
V \ar[r]^{\alpha} &  W_s \ar@{^{(}->}[r] &  W_{s-1} 
\ar@{^{(}->}[r] & \cdots \ar@{^{(}->}[r] & W_1 \ar@{^{(}->}[r] & W.
}
$$  
By Corollary~\ref{genlink}, applied to the group
$G^{D_{\lda}^k (\mbf W)}\le \Aut (D_{\lda}^k (\mbf W))$ and to the set
$A_{\Cal Y}^{D_{\lda}^k (\mbf W)} \subseteq \End(D_{\lda}^k (\mbf W))$, 
\begin{equation}\label{indlem1}
\gamma\left(G^{Q_{\lda}^k (\mbf W)}_{\vph Y}, 
A^{Q_{\lda}^k (\mbf W)}_{\Cal Y}\right) = 
\sum_{\alpha\in\Xi} \gamma\left(G^{L^{\alpha}}_{\vph Y}, A^{L^{\alpha}}_{\Cal Y}
\right).
\end{equation}
(Here we use the fact that, if $\mbf X\in \End(Q_{\lda}^k (\mbf W))$
and the actions of $\mbf X$ on $V$ and $W_s$ are 
$\Cal Y$-endomorphisms, then so is the action of $\mbf X$ on $Z$.) 

If $\alpha=0$, then, obviously,
$$\gamma\left(G^{L^{\alpha}}_{\vph Y}, A^{L^{\alpha}}_{\Cal Y} \right) =
\gamma\left(G^{D_{\lda}^k (\mbf W)}_{\vph Y}, A^{D_{\lda}^k (\mbf W)}_{\Cal Y} 
\right). 
$$
Since the quiver representation $D_{\lda}^k (\mbf W)$ is the disconnected union of 
$Q_{\lda}(\mbf W)$ and the quiver representation that consists of the space $V$, we have  
\begin{equation}\label{indlem2}
\gamma \!\left(G^{L^{0}}_{\vph Y}, A^{L^{0}}_{\Cal Y}\right) =
c(k,\Cal Y) \gamma \!\left(G^{Q_{\lda}(\mbf W)}_{\vph Y}, 
A^{Q_{\lda}(\mbf W)}_{\Cal Y} \right).
\end{equation}
Now consider the case $\alpha\ne 0$. Then $\alpha=\alpha_{W_{s+1}}$ for some 
$W_{s+1}\in F_{\lda'}$. 
Let $O$ be the quiver representation $I(K(L^{\alpha},e),e)$ where
$e$ is the arrow from $V$ to $W_s$ in $L^{\alpha}$. Then $O$  may be depicted as
$$
\xymatrix{
V' \ar@{^{(}->}[r] & V \ar@{->>}[r] & W_{s+1} \ar@{^{(}->}[r] & 
W_s \ar@{^{(}->}[r] & \cdots \ar@{^{(}->}[r] & W_1 \ar@{^{(}->}[r] & W,
}
$$
where $V'=\ker\alpha$ and the map $V \thra W_{s+1}$ is induced by $\alpha$. 
By Lemma \ref{KI}, the $G^{L^{\alpha}}$-orbits on
$A^{L^{\alpha}}_{\Cal Y}$ are in a one-to-one correspondence with the 
$G^O$-orbits on $A^O_{\Cal Y}$. However, renaming $V$ as $Z$ and $V'$ as $V$, 
we may identify $O$ 
with $Q_{\lda'}^{k-\lda_{s+1}}(\mbf W)$ where 
$\lda'=(\lda_1,\ldots,\lda_s,\lda_{s+1})$. Thus, 
\begin{equation}\label{indlem3}
\gamma\left(G^{L^{\alpha}}_{\vph Y}, A^{L^{\alpha}}_{\Cal Y} \right)=
\gamma\left(G^{Q_{\lda'}^{k-\lda_{s+1}}(\mbf W)}_{\vph Y},
A^{Q_{\lda'}^{k-\lda_{s+1}}(\mbf W)}_{\Cal Y} \right).
\end{equation}
The result now follows from~\eqref{indlem1}, \eqref{indlem2} 
and~\eqref{indlem3}.
\end{proof}

The main result of this paper may be stated in a very general form as follows.
\begin{thm}\label{gensq} 
Let $\Cal Y$ be a class of linear operators over $\mF_q$. 
Let $k\in \mN$. Suppose that $G$ is a subgroup of $\GL(W)$ and that 
$A\subseteq \End(W)\cap \Cal Y$ is preserved by $G$. Then
$$\gamma(G^{(k)}_{\vph Y}, A^{(k)}_{\Cal Y}) =
\sum_{j=0}^k c(k-j,\Cal Y) \sum_{\lda\in\Cal P_j^m} \xi_{\lda} (G,A).
$$
\end{thm}
\begin{proof} Recall that 
$l(\lda)$ is the length (the number of parts) of a partition $\lambda$.
If $s$ is a nonnegative integer, let 
\begin{eqnarray*}
a_s & = & \sum_{j=0}^k \sum_{\substack{\lda\in\Cal P_j^m \\ l(\lda)=s }}
c(k-j,\Cal Y) \xi_{\lda} (G,A), \\
b_s & = & \sum_{j=0}^k \sum_{\substack{\lda\in\Cal P_{k-j}^m \\ l(\lda)=s}}
\sum_{\mbf W\in F_{\lda}} 
\gamma(G^{Q_{\lda}^j (\mbf W)}, A^{Q_{\lda}^j (\mbf W)}_{\Cal Y}).  
\end{eqnarray*}
Since $A\subseteq \Cal Y$, we have  
$A_{\Cal Y}^{Q_{\lda}(\mbf W)}=A^{Q_{\lda}(\mbf W)}_{\vph Y}$.
Thus, by~\eqref{xi}, 
\begin{equation}\label{xi2}
\xi_{\lda}(G,A)=\sum_{\mbf W\in F_{\lda}}
\gamma\left(G^{Q_{\lda}(\mbf W)}_{\vph Y}, A^{Q_{\lda}(\mbf W)}_{\Cal Y}\right).
\end{equation}

Observe that 
$$b_0=\gamma\left(G^{(k)}_{\vph Y}, A^{(k)}_{\Cal Y} \right).$$
Indeed, the only non-zero summand of $b_0$ corresponds to the case $j=k$, 
$\lda=()$. 
Applying Lemma \ref{indlem} and rearranging the sums, we infer that, 
for any $s$,
\begin{eqnarray}
b_s & = & \sum_{j=0}^k \sum_{\substack{\lda\in\Cal P_{k-j}^m \\ l(\lda)=s}}
\sum_{\mbf W\in F_{\lda}} 
\gamma\left(G^{Q_{\lda}^j (\mbf W)}_{\vph Y}, A^{Q_{\lda}^j (\mbf W)}_{\Cal Y}\right) \nonumber\\
& = & \sum_{j=0}^k \sum_{\substack{\lda\in\Cal P_{k-j}^m \\ l(\lda)=s}}
\sum_{\mbf W\in F_{\lda}} 
c(j,\Cal Y)\gamma\!\left(G^{Q_{\lda}(\mbf W)}_{\vph Y}, A^{Q_{\lda}(\mbf W)}_{\Cal Y}\right) \nonumber\\
& + & \sum_{i=0}^k \sum_{\substack{\lda'\in\Cal P_{k-i}^m \\ l(\lda')=s+1}}
\sum_{\mbf W'\in F_{\lda'}} 
\gamma\left(G^{Q_{\lda'}^i (\mbf W)}_{\vph Y}, A^{Q_{\lda'}^i (\mbf W)}_{\Cal Y}\right) \nonumber\\
& = & a_s+b_{s+1}. \label{gensq1}
\end{eqnarray} 
(The last equality follows from~\eqref{xi2}.)
Note that, if $s>k$, then $a_s=b_s=0$ because any partition $\lda$ with 
$l(\lda)=s$ satisfies $|\lda|\ge s>k$. Hence, by~\eqref{gensq1},
$$b_0=a_0+a_1+\cdots+a_k,$$ 
and the result follows. 
\end{proof} 

We now deduce Theorem~\ref{pargen}. Let $k,m \in \mN$, and let $W$ be a vector space
over $\mF_q$ of dimension $m$. Consider the quiver representation
$Q^{(k)}=(E_0,E_1,\mbf U,\bs\alpha)$ given by \eqref{Qk}: say, $E_0=\{a,b,c\}$,
$U_a=V$, $U_b=Z$, $U_c=W$ ($\dim V=k$). Put $G=\GL(W)$ and $A=\Cal Y\cap \End(W)$.  
Then, for any partition $\lda\in \Cal P^m$, the set $F_{\lda}$ consists of just one flag, so 
\begin{equation}\label{pargenpr1}
\xi_{\lda}(G,A)=\kappa_{\lda}^m (\Cal Y).
\end{equation}
Moreover, $\mbf X=(X_a,X_b,X_c)\mapsto X_b$ induces a one-to-one correspondence between
the $G^{(k)}$-orbits in $A^{(k)}$ and the $\scr P(Z;V)$-orbits in 
$\Cal Y\cap \End(Z;V)$. 
(Here we identify $V$ with its image under the injective 
map $V\ra Z$.) Hence, 
\begin{equation}\label{pargenpr2}
\gamma\left(G^{(k)}_{\vph Y},A^{(k)}_{\Cal Y}\right)=\kappa_{(k)}^{k+m} (\Cal Y).
\end{equation}
Theorem~\ref{pargen} follows immediately from~\eqref{pargenpr1}, \eqref{pargenpr2} and Theorem~\ref{gensq}.

Let $\lda$ be a partition. It can be represented as
$$\lambda=(\underbrace{s_1,s_1,\ldots,s_1}_{u_1},
\underbrace{s_2,\ldots,s_2}_{u_2}, \ldots, \underbrace {s_l,\ldots,s_l}_{u_l})
$$
where $s_1>s_2>\ldots>s_l$. Let $\bar{\lambda}$ be the set $\{s_1,\ldots,s_l\}$.
If $S\subset \mN$ is a finite set and $k\in \mN$, let $r(k,S)$ be the number of
partitions $\lda$ such that $\bar{\lda}=S$. That is, $r(k,S)$ is the number of 
partitions $\lda$ such that $\lda_i\in S$ for all $i$ and, for each $s\in S$, there
exists $i$ such that $\lda_i=s$.

If $S=\{s_1,\ldots,s_l\}$ is a set and $s_1>\cdots>s_l$, 
we shall identify the set $S$ with the partition
$(s_1,s_2,\ldots,s_l)$. 
(So the notation $\xi_S (G,A)$ makes sense, for example.) 

Let $W$, $G\le \GL(W)$ and $A\subseteq \End(W)\cap \Cal Y$ be as above, with $m=\dim W$. 
It is clear that, if $\lambda\in\Cal P^m$, then
\begin{equation}\label{coll1}
\xi_{\lda}(G,A)=\xi_{\bar\lda} (G,A).
\end{equation}
(Indeed, duplicating subspaces  
in a flag $\mbf W$ does not add anything new to the structure.)
Moreover, if a $S=S'\sqcup \{m\}$ is a subset of $[1,m]$ containing $m$, then
\begin{equation}\label{coll2}
\xi_S (G,A)=\xi_{S'} (G,A).
\end{equation} 
These identities allow us to simplify the expression in Theorem~\ref{gensq}.
\begin{cor}\label{sqcor1} Under the hypotheses of Theorem~\ref{gensq},
\begin{equation*}
\gamma(G^{(k)}, A^{(k)}_{\Cal Y})  = 
\sum_{j=0}^k c(k-j,\Cal Y)
\sum_{S\subseteq [1,m-1]} (r(j,S)+r(j,S\cup \{m\})) \xi_S (G,A). 
\end{equation*}
\end{cor}

\begin{remark} When calculating the last sum, we only need to consider those
 $S\subseteq [1,m-1]$ for which $\sum_{s\in S} s\le k$: 
if $\sum_{s\in S} s>k$, then 
$r(j,S)=r(j,S\cup \{m\})=0$ for all $j\in [0,k]$. 
\end{remark}

Similarly, Theorem \ref{pargen} implies the following.
\begin{cor}\label{sqcor2} Let $k,m\in \mN$. Then
\begin{equation*}
 \kappa_{(k)}^{k+m} (\Cal Y)  = 
\sum_{j=0}^k c(k-j,\Cal Y)
 \sum_{S\subseteq [1,m-1]} (r(k,S)+r(k,S\cup \{m\})) \kappa_S^{m} (\Cal Y). 
\end{equation*}
\end{cor}

\section{Inverting the formula}\label{inv}

From now on, we shall assume that $\Cal Y$ is  
the class $\Cal N$ of all nilpotent endomorphisms, 
so $c(j,\Cal Y)=p(j)$ for all $j$. Let $W$ be an $m$-dimensional vector 
space over $\mF_q$, where $m\in \mN$ is fixed throughout the section. Let $G$ be a subgroup of $\GL(W)$ 
preserving a subset $A$ of $N(W)$ (as in Section~\ref{indarg}). If
$k\in \mZ_{\ge 0}$, define
\begin{eqnarray*}
\phi_k (G,A) & = & \sum_{S\subseteq [1,m]} r(k,S) \xi_S (G,A) \quad\text{and } \\
\psi_k (G,A) & = & \phi_k (G,A)-\phi_{k-m}(G,A)
\end{eqnarray*}
where, by convention, $\phi_k (G,A)=0$ if $k<0$.
If $S\subseteq [1,m-1]$, then 
$$
r(k,S \cup \{m\})=r(k-m,S) + r(k-m,S\cup \{m\})
$$ 
for all $k\in \mN$.
Also, by~\eqref{coll2}, $\xi_S (G,A)= \xi_{S\cup \{m\}} (G,A)$. Hence,
\begin{eqnarray*}
\phi_{k-m}(G,A)  & = & 
\sum_{S\subseteq [1,m-1]} r(k-m,S) \xi_S (G,A) \\
& + & \sum_{S\subseteq [1,m-1]} r(k-m,S \cup \{m\})\xi_{S \cup \{m\}} (G,A) \\
& = & \sum_{S\subseteq [1,m-1]} r(k, S\cup \{m\}) \xi_{S\cup\{m\}}(G,A).
\end{eqnarray*}
It follows that 
\begin{equation}\label{psi}
\psi_k (G,A) = \sum_{S\subseteq [1,m-1]} r(k,S) \xi_S (G,A).
\end{equation}

We now use Corollary~\ref{sqcor1} show that the numbers 
$\phi_k(G,A)$ and
$\psi_k (G,A)$ may be expressed (independently of $G$ and $A$) in terms of $\gamma(G^{(r)}_{\vph N},A^{(r)}_{\Cal N})$, $r\in [0,k]$.
\begin{lem}\label{expphi} 
Let $k\in \mN$. There exist integers $a_0,a_1,\ldots,a_k$  
that depend only on $k$ and $m$ (but not on $G$ or $A$) such that 
$$
\phi_k (G,A) = \sum_{j=0}^k a_j \gamma(G^{(j)}_{\vph N},A^{(j)}_{\Cal N}).
$$
\end{lem}
 
\begin{proof} We prove the lemma by induction on $k$. 
Observe that 
$$
\phi_0 (G,A) = \xi_{\varnothing}(G,A) = \gamma(G^{(0)}_{\vph N}, A^{(0)}_{\Cal N})
= \gamma(G,A).
$$
By Corollary~\ref{sqcor1},
\begin{eqnarray}
\gamma(G^{(k)}_{\vph N},A^{(k)}_{\Cal N}) & = & \sum_{j=0}^k p(k-j) \phi_j (G,A), \qquad 
\text{so } 
\label{usephi}\\
\phi_k (G,A) & = & \gamma(G^{(k)}_{\vph N},A^{(k)}_{\Cal N})-p(k)\gamma(G,A)- 
\sum_{j=1}^{k-1} p(k-j) \phi_j (G,A).\nonumber
\end{eqnarray}
The result for $\phi_k (G,A)$ now follows by the inductive hypothesis. 
\end{proof}

\begin{cor}\label{exppsi}
Let $k\in \mN$. There exist integers $a_0,a_1,\ldots,a_k$  
that depend only on $k$ and $m$ (but not on $G$ or $A$) such that 
$$
\psi_k (G,A) = \sum_{j=0}^k a_j \gamma(G^{(j)}_{\vph N},A^{(j)}_{\Cal N}).
$$
Also, there exist integers $a'_0,a'_1,\ldots,a'_k$ depending only on $k$ and $m$
such that
$$
\gamma(G^{(k)}_{\vph N},A^{(k)}_{\Cal N}) = \sum_{j=0}^k a'_j \psi_j (G,A).
$$
\end{cor}

\begin{proof} The first statement follows from Lemma~\ref{expphi}. 
The second statement follows from~\eqref{usephi} and the 
identity 
$$
\phi_j (G,A)= \sum_{i=0}^{\infty} \psi_{j-im}(G,A),
$$
which is a consequence of the definition of $\psi_k (G,A)$. 
\end{proof} 

Let $n=m(m-1)/2$.
It is proved in the Appendix that there exist (explicitly defined) 
integers $c_1,\ldots,c_n$, depending only on $m$, such that 
\begin{equation}\label{cn}
r(k,S)=-\sum_{j=k-n}^{k-1} c_{k-j} r(j,S).
\end{equation}
for all $k>n$ and all $S\subseteq [1,m-1]$.

\begin{prop}\label{expgen} Let $n=m(m-1)/2$. Let $k>n$. Then there exist
integers $a_{k0}, a_{k1},\ldots,a_{kn}$ depending only on $k$ and $m$ (but
 not on $G$ or $A$) such that
$$
\gamma(G^{(k)}_{\vph N},A^{(k)}_{\Cal N}) =
\sum_{j=0}^n a_{kj} \gamma(G^{(j)}_{\vph N},A^{(j)}_{\Cal N}).
$$
\end{prop}

\begin{proof}  
By~\eqref{psi} and~\eqref{cn}, for all $k>n$,
$$
\psi_k (G,A)= - \sum_{j=k-n}^{k-1} c_{k-j} \psi_j (G,A).
$$
An induction on $k$ shows that, for $k>n$, there exist
integers $a'_{k1},\ldots,a'_{kn}$ such that 
$$
\psi_k (G,A) = \sum_{j=1}^n a'_{kj} \psi_j (G,A).
$$ 
By Corollary~\ref{exppsi}, the result follows. 
\end{proof}

\begin{remark} The matrix formed by the numbers $r(k,S)$ as $k$ runs 
through $\mN$ and $S$ runs through non-empty subsets of $[1,m-1]$ is of rank
$m(m-1)/2=n$, as shown in the Appendix. Thus, Proposition~\ref{expgen}
is the best result that may be obtained via the method described.
\end{remark}

We now show that Theorem~\ref{rec} is a particular case of this result.
Let $\mbf l=(l_1,\ldots,l_s)$ be a tuple of nonnegative integers with
$l_1+\cdots+l_s=m$. As before, let $W=\mF_q^m$. 
Put $G=P^{\mbf l}(q)\le \GL(W)$ and $A=N^{\mbf l}(q)\subseteq \End(W)$. Then 
$A$ is preserved by $G$. The quiver representation $Q^k$ may be depicted as 
$$
\xymatrix{  
U_a \ar@{^{(}->}[r] & U_b \ar@{->>}[r] &  W=U_c
}
$$
where the sequence is exact.
Then the map $\End(Q^k)\ra \End(U_b)$, $\mbf X\mapsto X_b$ establishes an
isomorphism between the $G^{(k)}$-action on $A^{(k)}_{\Cal N}$ and the 
$P^{k,\mbf l}(q)$-action on $N^{k,\mbf l}(q)$. (Here we identify $P^{k,\mbf l}(q)$
 with the subgroup of $\GL(U_b)$ consisting of those maps $g$ that preserve
the image of $U_a$ and act as an element of $G$ on $W$.) Thus,
\begin{equation}\label{rhoquiver}
\rho_{k,\mbf l}(q) = \gamma(G^{(k)}_{\vph N},A^{(k)}_{\Cal N}).
\end{equation}
Theorem~\ref{rec} now follows from Proposition~\ref{expgen}.

With $n=m(m-1)/2$, as before, 
it is shown in the Appendix that there exist (explicitly defined) 
integers $d_1,\ldots,d_n$ such that 
\begin{equation}\label{djeq}
\sum_{j=1}^n d_j r(j,S) =  
\begin{cases}
1 & \text{if } S=[1,m-1], \\ 
0 & \text{if } S \subsetneq [1,m-1].
\end{cases}
\end{equation}

\begin{prop}\label{maxredgen} Let $m\in \mN$, and let $W=\mF_q^m$.
Let $n=m(m-1)/2$. Then there exist integers
$a_{0},a_1,\ldots,a_n$ depending only on $m$ such that, for any 
$G\le \GL(W)$ and $A\subseteq \N(W)$ with $G$ preserving $A$, 
$$   
\xi_{[1,m-1]}(G,A)=\sum_{j=0}^n a_j \gamma(G^{(j)}_{\vph N},A^{(j)}_{\Cal N}).
$$
\end{prop}

\begin{proof} By~\eqref{djeq} and~\eqref{psi},
$$
\sum_{j=1}^n d_j \psi_j (G,A) = \xi_{[1,m-1]}(G,A).
$$
The result now follows from Corollary~\ref{exppsi}.
\end{proof}

In order to deduce Theorem~\ref{impl}, we put $G=\GL_m (q)$ and $A=\N_m (q)$. 
Then, by~\eqref{rhoquiver}, $\rho_{(j,m)}(q)=\gamma(G^{(j)},A^{(j)})$. 
Also, $\xi_{[1,m-1]}(G,A)=\rho_{(1^m)}(q)$. Thus, Theorem~\ref{impl} follows from
Proposition~\ref{maxredgen}.

\section{Groups associated to preordered sets}\label{presets}

Let $C$ be a finite set. A binary relation $\preq$ on $C$ is a 
\emph{preorder} if 
\begin{enumerate}[(i)] 
\item $x\preq x$ for all $x\in C$; and 
\item if $x\preq y$ and $y\preq z$, then $x\preq z$, for all $x,y,z\in C$.
\end{enumerate}
A \emph{preordered set} is a finite set together with a preorder on it. 
Two elements $x,y\in C$ are said to be \emph{comparable} if either $x\preq y$
 or $y\preq x$. 
If $(C,\preq)$ is a preordered set, one can define an equivalence relation on
$C$ as follows: $x$ is equivalent to $y$ if and only if $x\preq y$ and 
$y\preq x$. The \emph{clots} of $C$ are, by definition, 
the equivalence classes with respect to
 this relation. 
 Say that a clot $D\subseteq C$ is \emph{minimal} if 
there is no $x\in C\setminus D$ such that $x\preq y$ for some $y\in D$. 
A \emph{partially ordered set} is a preordered set each of 
whose clots contains only one element.
The \emph{dual} of a preordered set $(C,\preq)$ is $(C,\preq')$ where
$x\preq' y$ if and only if $y\preq x$. We denote the dual of a preordered set
$C$ by $C^*$. If $(C_1,\preq_1)$ and $(C_2,\preq_2)$ are two preordered sets, their 
\emph{disjoint union} is the set $C_1 \sqcup C_2$ with the preorder $\preq$
defined as follows:
$x\preq y$ if and only if $x,y\in C_i$ and $x\preq_i y$ for some $i\in\{1,2\}$.
Where appropriate, $C_1 \sqcup C_2$ will denote the corresponding preordered set
rather than just a set.

We shall define preordered sets by diagrams as follows. Clots 
will correspond to nodes, with the number at each node equal to the number
 of elements in the corresponding clot; $x\preq y$ if and only if one can get 
from the node of $x$ to the node of $y$ by going along arrows. For 
example, 
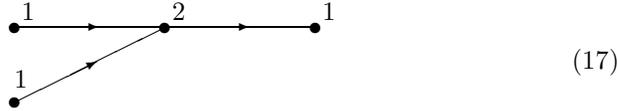
\begin{equation}\label{expre}
\begin{array}{c}
\begin{picture}(40,12)
\multiput(0,10)(20,0){3}{\node}
\put(0,0){\node}
\put(0,10){\usebox{\rvec}}
\put(20,10){\usebox{\rvec}}
\put(0,0){\usebox{\ruvec}}
\put(1,11){$1$}
\put(21,11){$2$}
\put(41,11){$1$}
\put(0,2){$1$}
\end{picture}
\end{array}
\end{equation}
is (isomorphic to) the preordered set $\{1,2,3,4,5\}$ where $\preq$ is defined as follows:
$1\preq 3\preq 4 \preq 3 \preq 5$ and $2\preq 3$.

Let $q$ be a prime power. Suppose $(C,\preq)$ is a (finite) preordered set.
Let $M^C(q)=M^{(C,\preq)}(q)$ be 
the set of all $C\times C$ matrices $X=(x_{ij})$ 
over $\mF_q$ such that $x_{ij}=0$ unless $i\preq j$, for all $i,j\in C$. 
Let $P^C (q)$ be the group of all invertible matrices in $M^C (q)$, and let
$N^C (q)$ be the set of all nilpotent matrices in $M^C (q)$. Let
$$
\rho_C (q) = \gamma(P^C (q), N^C (q)).
$$ 
For example, if $C$ is the preordered set given by~\eqref{expre}, 
then $M^{C}(q)$ 
is (up to a permutation of $C$) 
the set of matrices of the form
$$
\begin{pmatrix}
* & 0 & * & * & * \\
0 & * & * & * & * \\
0 & 0 & * & * & * \\
0 & 0 & * & * & * \\
0 & 0 & 0 & 0 & * \\
\end{pmatrix}.
$$

If $\mbf l=(l_1,\ldots,l_s)$ is a tuple of nonnegative integers, we shall
 also write $\mbf l$ for the preordered set
$$\begin{picture}(70,2)
\multiput(0,0)(20,0){2}{\node}
\multiput(50,0)(20,0){2}{\node}
\put(0,0){\usebox{\rvec}}
\put(50,0){\usebox{\rvec}}
\put(33,0){$\cdots$}
\put(1,1){$l_1$}
\put(21,1){$l_2$}
\put(51,1){$l_{s-1}$}
\put(71,1){$l_s$}
\end{picture}
$$
Note that then $P^{\mbf l}(q)$, $\rho_{\mbf l}(q)$, etc. are as 
defined previously.

If $C$ is a preordered set, we shall aim to reduce the problem of finding 
$\rho_C (q)$ to orbit-counting problems for matrices of size less than $|C|$,
 using Theorem~\ref{gensq}.
In some cases, we shall be able to express $\rho_C (q)$ in terms of 
$\rho_{O}(q)$ where $O$ 
varies among preordered sets of size less than
$|C|$. This will allow us to calculate $\rho_C (q)$ using recursion for some 
$C$. 

Let $(C,\preq)$ be a finite preordered set. 
Let $D\subseteq C$ be a minimal clot in $C$ with $|D|=k$, say.
Let $E=\{x\in C: y\not\preq x \; \forall y\in D \}$. Let $\bar{D}=C\setminus D$, 
$\bar E=C\setminus E$ and $C'=C\setminus (D\cup E)$. We shall view 
$C\times C$ matrices as endomorphisms of a vector space $V$ over $\mF_q$
equipped with a basis $\{e_i\}_{i\in C}$. Let 
$V_D= \Span\{e_i\}_{i\in D}$, $V_E=\Span\{e_i\}_{i\in E}$, 
$V_{\bar D}=V/V_D$ and $V_{C'}=V/(V_D+V_E)$. We may view  $M^{\bar D}(q)$
  as a subring of $\End(V_{\bar D})$ 
using the projections of $e_i$, 
$i\in \bar D$ as a basis of $V_{\bar D}$. Similarly, we may view
$M^{C'}(q)$ as a subring of $\End (V_{C'})$. 
(By abuse of notation, we shall view $\{e_i\}_{i\in C'}$ as a basis of $V_{C'}$.)   
Let $m=|C'|$. If 
$$
S=\{s_1>\cdots>s_r\}\subseteq [1,m-1],
$$ 
let $\Cal H_S=\Cal H_S (C')$ be the set of all flags $\mbf W=(W_1,\ldots, W_r)$ 
such that
$$
W_r \le W_{r-1} \le \cdots \le W_1 \le V_{C'}
$$ 
and $\dim W_i=s_i$ for all $i$. 
Let $\pi: V_{\bar D}\thra V_{C'}$ be the natural projection.
If $\mbf W\in \Cal H_S$, write $\pi^{-1}(\mbf W)$ for the flag
$(\pi^{-1}(W_1),\ldots,\pi^{-1}(W_r))$ in $V_{\bar D}$.
Let $M^{\bar D}_{\mbf W}(q) = \scr P(V_{\bar D}; \pi^{-1}(\mbf W))$.
That is, $M^{\bar D}_{\mbf W}(q)$ is the ring of all elements of $M^{\bar D}(q)$
whose action on $V_{C'}$ preserves $\mbf W$.
Let $P^{\bar D}_{\mbf W}(q)$ be the group of all invertible elements 
in $M^{\bar D}_{\mbf W}(q)$, and let $N^{\bar D}_{\mbf W}(q)$ be the set of all nilpotent
elements in $M^{\bar D}_{\mbf W}(q)$.   
  
\begin{prop}\label{preord} With the notation as in the previous paragraph,
suppose that
$H_S$ is a complete set of representatives of the $P^{C'}(q)$-orbits on 
$\Cal H_S$ for each $S\subseteq [1,m-1]$. Then
$$
\rho_C (q)  =  \sum_{j=0}^k p(k-j) 
\sum_{S\subseteq [1,m-1]} (r(j,S)+r(j,S\cup \{m\}))  
 \sum_{\mbf W\in H_S} 
\gamma(P^{\bar D}_{\mbf W}(q), N^{\bar D}_{\mbf W} (q)). 
$$
\end{prop}

\begin{proof} Let $Z$ be a complete set of representatives of $P^{\bar D}(q)$-orbits on
$N^{\bar D}(q)$. Let 
$\pi_{D}: M^C (q) \ra M^{\bar D}(q)$ be the natural projection. 
If $X\in Z$, let $G_X = \pi_D^{-1}(\Stab_{P^{\bar D}(q)} (X)) \cap P^{C}(q)$ 
and $A_X = \pi_D^{-1}(X) \cap N^C (q)$.
 By Lemma~\ref{actprod},
\begin{equation}\label{preord1}
\rho_C (q) = \gamma(P^C (q), N^C (q))= \sum_{X\in Z} \gamma(G_X, A_X). 
\end{equation}
Let $\pi_E: M^C (q) \ra M^{\bar E}(q)$ be the natural projection. Let $X\in Z$.
It is easy to see that $\pi_E$ establishes a bijection between the 
$G_X$-orbits on $A_X$ and the $\pi_E (G_X)$-orbits on $\pi_E (A_X)$. So
\begin{equation}\label{preord2}
\gamma(G_X, A_X)= \gamma(\pi_E (G_X), \pi_E (A_X))
\end{equation}
Let $Q$ be the quiver representation
$$
\xymatrix{
V_D \ar@{^{(}->}[r] &  V_{\bar E} \ar@{->>}[r] &  V_{C'}
}
$$
where the first map is induced by the inclusion $V_D \hra V_{C}$ and the second map is the natural projection. Then $Q$ is isomorphic to $Q^k$ in the notation
 of Section~\ref{indarg} ($V_{C'}$ plays the role of $W$).  
Let $\chi: \End(Q) \ra \End(V_{\bar E}; V_D)$ be the isomorphism which maps an 
endomorphism $\mbf Y$ of $Q$ to the action of $\mbf Y$ on $V_{\bar E}$.

Our aim is to express $\gamma(\pi_E (G_X), \pi_E (A_X))$ by applying 
Corollary~\ref{sqcor1} to the quiver $Q^k$. In fact, this expression is the only substantial step of the proof.

Let $G'_X\le P^{C'}(q)$ be the image of $G_X$ under the natural map 
$P^C (q) \ra P^{C'}(q)$.
Let $\eta: M^{\bar D}(q)\ra M^{C'}(q)$ be the natural projection, and let  
 $A'_X=\{\eta(X)\}\subseteq N^{C'}(q)$. 
Let 
$\omega: \End(V_{\bar E}; V_D)\ra \M_{C',C'}(q)$
 be the natural projection,  $(y_{ij})_{i,j\in \bar E}\mapsto (y_{ij})_{i,j\in C'}$.
Since $d\preq c$ for all $d\in D$ and $c\in C'$, 
$$
\omega^{-1}(M^{C'}(q))=M^{\bar E}(q).
$$
We identify the quiver representation $Q$ with $Q^k$; let $(G'_X)^{(k)}$ and 
$(A'_X)^{(k)}_{\Cal N}$ be as defined in Section~\ref{indarg}.
Then, by those definitions,
\begin{eqnarray*}
\chi((G'_X)^{(k)})& = &\omega^{-1}(G'_X)\cap P^{\bar E}(q) \quad \text{and} \\
\chi((A'_X)^{(k)}_{\Cal N}) & = & \omega^{-1}(A'_X)\cap N^{\bar E}(q).
\end{eqnarray*}
Hence, clearly, 
$\chi((A'_X)^{(k)}_{\Cal N}) = \pi_E (A_X)$.  

We claim that $\chi((G'_X)^{(k)}) = \pi_E (G_X)$.
It is clear that $\pi_E (G_X)\subseteq \omega^{-1}(G'_X) \cap P^{\bar E}(q)$. For the 
converse, suppose that $Y=(y_{ij})\in P^{\bar E}(q)$ satisfies $\omega(Y)\in G'_X$.
Then there exists a matrix $B=(b_{ij})\in M^C (q)$ such that 
$\pi_D (B)$ fixes $X$ and $b_{ij}=y_{ij}$ for all $i,j\in C'$. Define 
$R=(r_{ij})\in M^{C}(q)$ by `gluing together' $\pi_D (B)$ and $Y$: 
$r_{ij}=y_{ij}$ if $i,j \in D\cup C'$ 
and $r_{ij}=b_{ij}$ if $i,j\in E\cup C'$. Then $\pi_D (R)=\pi_D (B)$, so $R\in G_X$, 
and $\pi_E (R)=Y$. Hence, $Y\in \pi_E (G_X)$.     

We are now in a position to apply 
Corollary~\ref{sqcor1}:
\begin{eqnarray}
\gamma(\pi_E (G_X), \pi_E (A_X)) & = & \gamma((G'_X)_{\vph N}^{(k)}, (A'_X)_{\Cal N}^{(k)}) \nonumber\\ 
 & &  
 \!\!\!\!\!\!\!\!\!\!\!\!\!\!\!\!\!\!\!\!\!\!\!\!\!\!\!\!\!\!\!\!\!\!\!\!\!\!\!\!\!
  \!\!\!\!\!\!\!\!\!\!\!\!\!\!\!\!\!\!\!\!\!\!\! = \;  
 \sum_{j=0}^k p(k-j) 
\sum_{S\subseteq [1,m-1]} (r(j,S)+r(j,S\cap \{m\})) \xi_S (G'_X, A'_X) \label{preord3}
\end{eqnarray}  
By definition, 
\begin{eqnarray*}
\xi_S (G'_X,A'_X) & = &\gamma(G'_X, \{\mbf W\in \Cal H_{S}: 
X(\pi^{-1}(\mbf W))= \pi^{-1} (\mbf W) \}) \\
& = & \gamma(\Stab_{P^{\bar D}(q)}(X), \{\mbf W\in \Cal H_{S}: X(\pi^{-1}(\mbf W))= 
\pi^{-1}(\mbf W) \} ).
\end{eqnarray*} 
Let 
$$
L=\{ (X,\mbf W)\in N^{\bar D}(q)\times \Cal H_S: X(\pi^{-1}(\mbf W)) = 
\pi^{-1}(\mbf W) \}
$$
By Lemma~\ref{actprod} (applied in two different ways),
\begin{equation}
\sum_{X\in Z} \xi_S (G'_X, A'_X )  
  =   \gamma(P^{\bar D}(q), L) 
  =   \sum_{\mbf W\in H_S} \gamma(P^{\bar D}_{\mbf W} (q), N^{\bar D}_{\mbf W}(q)). 
\label{preord4}      
\end{equation}
The result follows by combining~\eqref{preord1}, \eqref{preord2}, 
\eqref{preord3} 
and~\eqref{preord4}.
\end{proof}

Therefore, in order to compute $\rho_C (q)$, 
it is enough to find a set of representatives $H_S$ for 
each $S\subseteq [1,m-1]$ and to calculate 
$\gamma(P^{\bar D}_{\mbf W}(q), N^{\bar D}_{\mbf W}(q))$ for every 
$\mbf W\in H_S$. 

Suppose that any two elements of $C'$ are comparable. 
Then $C'$ is isomorphic to some $\mbf l=(l_1,\ldots,l_s)$. 
Let $S=\{t_1>\cdots>t_u\} \subseteq [1,m-1]$.
The orbits
 of $P^{\mbf l}(q)$ on $\Cal H_S (\mbf l)$ are well understood. (In fact, these
orbits correspond to orbits of pairs of flags, of appropriate dimensions, 
under the action of $\GL_m (q)$.)
Let $\Cal A_S^{\mbf l}$ be the set of tuples 
$\mbf n=(n_{ij})_{i\in [1,s], j\in [1,u]}$ of nonnegative integers satisfying
$$
l_i\ge n_{i1}\ge n_{i2} \ge \cdots \ge n_{iu} 
\quad \text{for all } i\in [1,s] \; \text{ and }
$$
$$
\sum_{i=1}^s n_{ij} = t_j \qquad\quad \text{ for all } j\in [1,u].
$$
Relabel the basis $\{e_i\}_{i\in C'}$ of $V_{C'}$ as
$\{f_{ij}\}_{i\in [1,s],j\in [1,l_i]}$ 
so that $P^{\mbf l}(q)=P^{C'}(q)$ is the stabiliser of the flag 
consisting of the subspaces $\Span \{f_{ij} \}_{i\in [1,r],j\in [1,l_i]}$, 
$r=1,\ldots,s$.
(This relabelling is given by an isomorphism between $C'$ and $\mbf l$.)
For each $\mbf n\in \Cal A_S^{\mbf l}$, define a flag 
$\mbf W^{\mbf n}\in \Cal H_S (\mbf l)$ as follows:
$$
W^{\mbf n}_a = \Span \{ f_{ij} : \: 1\le i \le s,\: 1\le j \le n_{ia} \},
\quad a=1,2,\ldots u.
$$

\begin{lem}\label{doublefl} \emph{\cite[Example 2.10]{MWZ}} Let 
$\mbf l=(l_1,\ldots,l_s)$ be a tuple of nonnegative integers with
$m=l_1+\cdots+l_s$. Let $S\subseteq [1,m-1]$. Then 
$\{\mbf W^{\mbf n}: \mbf n\in \Cal A_S^{\mbf l} \}$ is a complete set of
representatives of 
$P^{\mbf l}(q)$-orbits on $\Cal H_S (\mbf l)$.
\end{lem}

If $C'$ is isomorphic to some $\mbf l$, we choose 
a set of representatives $H_S$ as given
 by Lemma~\ref{doublefl}. 
In particular, for every $\mbf W\in H_S$, each $W_j$ is spanned 
by a subset of the standard basis $\{e_i\}_{i\in C'}$. 
 
Let $\mbf n \in \Cal A_S^{\mbf l}$. Write $n_{i0}=l_i$ and $n_{i,u+1}=0$ for all $i\in [1,s]$. 
Define a new preorder $\preq'$ on $\bar D$ as follows: 
\begin{enumerate}[(i)]
\item\label{neword1} suppose $i_1,i_2 \in C'$; 
let $a_1,b_1,a_2,b_2$ be such that
$e_{i_1}=f_{a_1,b_1}$ and $e_j=f_{a_2,b_2}$; then 
$i_1\preq' i_2$ if and only if $a_1 \le a_2$ and there exists 
$c\in [0,u]$ such that $b_1 \le n_{a_1,c}$ and $b_2> n_{a_2,c+1}$; 
\item if either $i\notin C'$ or $j\notin C'$, then $i\preq' j$ if and only 
if $i\preq j$. 
\end{enumerate}
Note that, in case~\eqref{neword1}, $i_1\preq i_2$ if and only if $a_1\le a_2$.
It is straightforward to check that, if $i_1,i_2\in C'$, then $i_1\preq' i_2$ if and 
only if $i_1\preq i_2$ and, for all $j\in [1,u]$, $e_{i_1}\in W^{\mbf n}_j$ implies 
$e_{i_2}\in W^{\mbf n}_j$. (In fact, this is our motivation for defining $\preq'$).

Let $O=O(S,\mbf n)$ be the preordered set $(\bar D,\preq')$. It follows that
$M^{\bar D}_{\mbf W^{\mbf n}}(q)=M^{O(S,\mbf n)}(q)$. Hence,
$$
\gamma(P^{\bar D}_{\mbf W^{\mbf n}}(q), N^{\bar D}_{\mbf W^{\mbf n}}(q)) 
= \rho_{O(S,\mbf n)}(q),
$$
and we deduce the following from Proposition~\ref{preord}.

\begin{cor}\label{preordcor} 
Let $(C,\preq)$ be a preordered set with a minimal 
clot $D$. Let $k=|D|$. Let $C'=\{x\in C: y\preq x \;\forall y\in D\}\setminus D$. Let $m=|C'|$. Suppose
the preordered set $C'$ is isomorphic to some $\mbf l=(l_1,\ldots,l_s)$. Then
$$
\rho_C (q)  =  \sum_{j=0}^k p(k-j) 
\sum_{S\subseteq [1,m-1]} (r(j,S)+r(j,S\cup \{m\}))  
 \sum_{\mbf n\in \Cal A_S^{\mbf l}} \rho_{O(S,\mbf n)} (q), 
$$ 
where $O(S,\mbf n)$ is defined in terms of $\bar D=C\setminus D$, $C'$ 
and $\preq$ as above.
\end{cor}

Since $|\bar{D}|<|C|$, this allows us to compute $\rho_C (q)$ by recursion as long as we can always find a minimal clot $D$ such that any two 
elements of $C'$ are comparable. Using this method, one may compute 
$\rho_{(l_1,\ldots,l_s)}(q)$ for all tuples $\mbf l$ with  $l_1+\cdots+l_s\le 6$,
thus proving Proposition~\ref{rhosmall} (the explicit expressions are given at the end of Section~\ref{dualrep}). Indeed, one can choose the clots
$D$ in such a way that the only preordered set occurring in that computation
for which Corollary~\ref{preordcor} is not applicable is the one considered
in Example~\ref{Delta2ex} below. For more detail on the computation, 
see~\cite[Appendix B]{thesis}. 
(The computation also uses the symmetry $\rho_C (q)=\rho_{C^*}(q)$ proved 
in Section~\ref{dualrep} below.) 

We now 
consider some examples of computations using Proposition~\ref{preord} and 
Corollary~\ref{preordcor}. For ease of notation, 
we list the flags $\mbf W=\mbf W^{\mbf n}\in H_S$ directly, without giving 
$\mbf n$.

\begin{exam}\label{ex1111} 
Suppose $C$ is isomorphic to $(1,1,1,1)$: say, $C=[1,4]$ with the 
usual order. The only minimal clot is $D=\{1\}$. Then $E=\varnothing$ and 
$C'=\{2,3,4\}$. The only non-zero terms in Corollary~\ref{preordcor} 
are given by
$S=\varnothing$ and $S=\{1\}$. If $S=\varnothing$, we get a summand 
$\rho_{(1^3)}(q)$. 
Assume $S=\{1\}$. Then 
any $\mbf W\in \Cal H_S$ consists of a single $1$-dimensional space $W_1$. 
So $H_S$ consists of $\langle e_2 \rangle$, 
$\langle e_3 \rangle$ and 
$\langle e_4 \rangle$. If $W_1=\langle e_2 \rangle$, 
then $O\simeq (1^3)$, so we get another 
term $\rho_{(1^3)}(q)$. If $W_1=\langle e_3 \rangle$, 
$O$ is isomorphic to the preordered set $\Delta_1$ given by the diagram
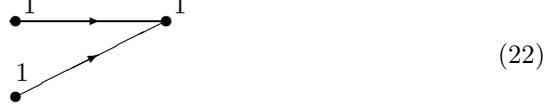
\begin{equation}\label{Delta1}
\begin{array}{c}
\begin{picture}(20,12)
\multiput(0,10)(20,0){2}{\node}
\put(0,0){\node}
\put(0,10){\usebox{\rvec}}
\put(0,0){\usebox{\ruvec}}
\put(1,11){1}
\put(21,11){1}
\put(0,2){1}
\end{picture}
\end{array}
\end{equation} 
\end{exam}
Finally, if $W_1=\langle e_4 \rangle$, then 
$O\simeq (1,1) \sqcup (1)$, so we get a term $\rho_{(1,1)}(q)$. Hence,
$$
\rho_{(1^4)}(q)= 2\rho_{(1^3)}(q)+\rho_{\Delta_1}(q)+\rho_{(1^2)}(q).
$$

\begin{exam} We compute $\rho_{\Delta_1}(q)$, where $\Delta_1$ is 
given by~\eqref{Delta1}. We may take for $D$ either of the minimal elements of 
$\Delta_1$. Then $E$ consists of the other minimal element, and $C'$ consists of the maximal element 
(label the elements $1,2,3$ so that $3$ is the maximal element). 
If $S=\varnothing$, we get $\rho_{(1^2)}(q)$. If $S=\{1\}$, then necessarily
$W_1=\langle e_3 \rangle$ and we get $\rho_{(1^2)}(q)$ again. Hence, 
$\rho_{\Delta_1}(q)=2\rho_{(1^2)}(q)=4$ for all $q$.
\end{exam}

\begin{exam}\label{Delta2ex} We now compute $\rho_{\Delta_2}(q)$ where $\Delta_2$ is the 
partially ordered set
$$
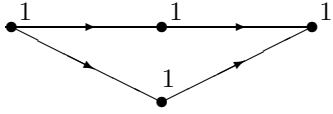
\begin{picture}(40,12)
\multiput(0,10)(20,0){3}{\node}
\put(20,0){\node}
\put(0,10){\usebox{\rvec}}
\put(20,10){\usebox{\rvec}}
\put(0,10){\usebox{\rdvec}}
\put(20,0){\usebox{\ruvec}}
\put(1,11){1}
\put(21,11){1}
\put(20,2){1}
\put(41,11){1}
\end{picture}
$$
\end{exam}
We may assume that the elements of $\Delta_2$ are $1,2,3,4$ 
where $1$ is the smallest element and $4$ is the largest. 
The only minimal clot is $D=\{1\}$. 
We have $C'=\{2,3,4\}$. 
In this case, $C'$ is not isomorphic to any $\mbf l$, so we 
apply Proposition~\ref{preord} directly. 
If $S=\varnothing$, we get the term $\rho_{\Delta_1}(q)$. Assume $S=\{1\}$. 
The following is a complete set of representatives for the action of 
$P^{C'}(q)$ on the $1$-dimensional subspaces $W_1$ in $V_{C'}$:
$$
\{ \langle e_2 \rangle,
\langle e_3 \rangle, \langle e_4 \rangle, \langle e_2+e_3 \rangle \}.
$$
In each of the first two cases, 
$\gamma(P^{\bar D}_{\mbf W}(q), N^{\bar D}_{\mbf W} (q))=\rho_{\Delta_1}(q)$. 
If $W_1=\langle e_4 \rangle$, then $N^{\bar D}_{\mbf W}(q)$ contains just the 
zero matrix, so the corresponding term in the sum is equal to $1$. 
Finally, suppose $W_1=\langle e_2+e_3 \rangle$. Then we do not get a term of the form 
$\rho_{O}(q)$ for a preordered set $O$. It 
is possible to apply Corollary~\ref{genlink} 
to find the corresponding term, 
but in this case we may find the orbits directly. 
With respect to the basis $\{e_2,e_3,e_4\}$, we have:
\begin{eqnarray*}
P^{C'}_{\mbf W} (q) & = & \left\{  
\begin{pmatrix}       
a & 0 & * \\
0 & a & * \\
0 & 0 & b 
\end{pmatrix}
: a,b \in \mF_q \setminus \{0\} \right\} 
 \quad \text{and} \\
N^{C'}_{\mbf W} (q) & = &
\left\{\begin{pmatrix}
0 & 0 & * \\
0 & 0 & * \\ 
0 & 0 & 0 
\end{pmatrix}\right\}.
\end{eqnarray*}
It is easy to see that a complete set of representatives of $P^{C'}_{\mbf W}(q)$-orbits
on $N^{C'}_{\mbf W}(q)$ is 
$$\left\{
\begin{pmatrix}
0 & 0 & 0 \\
0 & 0 & 0 \\
0 & 0 & 0 \\
\end{pmatrix},
\begin{pmatrix}
0 & 0 & 1 \\
0 & 0 & 0 \\
0 & 0 & 0 \\
\end{pmatrix},
\begin{pmatrix}
0 & 0 & a \\
0 & 0 & 1 \\
0 & 0 & 0
\end{pmatrix}
: a\in \mF_q
\right\}.
$$
So in this case $\gamma(P^{C'}_{\mbf W}(q),N^{C'}_{\mbf W}(q))=q+2$.
Thus,
$$
\rho_{\Delta_2}(q)=3\rho_{\Delta_1}(q)+ 1+ (q+2) = q+15.
$$

\begin{exam}\label{ex222} We consider $\rho_{(2,2,2)}(q)$. $(2,2,2)$ is isomorphic to 
$C=[1,6]$ with clots $\{1,2\}$, $\{3,4\}$ and $\{5,6\}$ (in the increasing order).
The only minimal clot is $D=\{1,2\}$, so $C'=\{3,4,5,6\}$. The only non-zero terms
 in Corollary~\ref{preordcor} come from $S=\varnothing$, $S=\{1\}$ and $S=\{2\}$.   
If $S=\varnothing$, we get $p(2)\rho_{(2,2)}=2\rho_{(2,2)}$. 
If $S=\{1\}$, $H_S$ contains two flags $\mbf W=(W_1)$, namely 
 $W_1=\langle e_3 \rangle$ and $W_1=\langle e_5 \rangle$. 
In the first case, we get $\rho_{(1,1,2)}(q)$. 
In the second case, $O$ is isomorphic to the preordered set $\Delta_3$ given by
$$
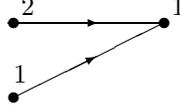
\begin{picture}(20,12)
\multiput(0,10)(20,0){2}{\node}
\put(0,0){\node}
\put(0,10){\usebox{\rvec}}
\put(0,0){\usebox{\ruvec}}
\put(1,11){2}
\put(21,11){1}
\put(0,2){1}
\end{picture}
$$
These two terms occur with the multiple $2$ in the sum because
$r(2,\{1\})=r(1,\{1\})=1$. Finally, if $S=\{2\}$, the possibilities for 
$W_1$ are: $\langle e_3,e_4 \rangle$, $\langle e_3,e_5 \rangle$ and 
$\langle e_5,e_6 \rangle$. These give rise to the terms $\rho_{(2,2)}(q)$,
$\rho_{\Delta_2}(q)$ and $4$ respectively. 
(In the last case, $O\simeq (2) \sqcup (2)$.) Thus,
$$
\rho_{(2,2,2)}(q)=3\rho_{(2,2)}(q) + 2\rho_{(1,1,2)}(q)+ 
2\rho_{\Delta_3}(q) + \rho_{\Delta_2}(q) + 4.   
$$ 
\end{exam}

\section{Dual quiver representations}\label{dualrep}

Let $Q=(E_0,E_1,\mbf U,\bs\alpha)$ be a quiver representation over a field
$\mF_q$. Define
the dual representation $Q^*=(E_0,E_1^*,\mbf U^*, \bs\alpha^*)$ as follows:
\begin{enumerate}[(i)]
\item $E_1^*$ is obtained by inverting all the arrows in $E_1$: for each 
$e\in E_1$, $\sigma^{*}(e)=\tau(e)$ and $\tau^*(e)=\sigma(e)$;
\item $\mbf U=(U_a^*)_{a\in E_0}$ where $U_a^*$ is the dual space to $U_a$;
\item $\bs\alpha^*=(\alpha_e^*)_{e\in E_1}$ where $\alpha_e^*$ is the dual 
map to $\alpha_e$.
\end{enumerate}
  
Recall that $\theta(Q)=\gamma(\Aut(Q),N(Q))$. If $\mbf X\in \End(Q)$, let
$\mbf X^*\in \End(Q^*)$ be the endomorphism $(X_a^*)_{a\in E_0}$. 

\begin{prop}\label{dualquiv} Let $Q=(E_0,E_1,\mbf U,\bs\alpha)$ be a quiver
 representation. The map $\mbf X\mapsto \mbf X^*$ is a ring isomorphism between
$\End(Q)$ and $\End(Q^*)^{\op}$. 
Hence,
$\theta(Q)=\theta(Q^*)$.  
\end{prop}

\begin{proof} Clearly, $\mbf X\mapsto \mbf X^*$ is a ring homomorphism from
$\End(Q)$ to $\End(Q^*)^{\op}$. If $V$ is a finite dimensional vector space, 
$V^{**}$ is naturally equivalent to $V$.
This equivalence establishes an isomorphism from $Q^{**}$ onto $Q$, which 
identifies $\mbf X^{**}$ and $\mbf X$ for all $\mbf X\in \End(Q)$. 
Hence, $\mbf X\mapsto \mbf X^*$ is a bijection.
\end{proof}

Now let $(C,\preq)$ be a preordered set. 
Let $V$ be a vector space over $\mF_q$ with a basis $\{e_i\}_{i\in C}$. 
As in the previous section, we may then identify $M^C (q)$ with
 a subring of $\End(V)$ using this basis.  

For each $i\in C$, let 
$\Cal D(i)=\{j\in C: j\preq i\}$. Let $T_C=T_C(q)$ be the quiver representation
$(C\sqcup\{0\},E_1,\mbf U,\bs\alpha)$ such that
\begin{enumerate}[(i)]
\item $E_1$ is equal to $C$ as a set, $\sigma(i)=i$ and $\tau(i)=0$ for 
all $i\in C$;
\item $U_0=V$, and $U_i=\Span \{e_j\}_{j\in \Cal D(i)}\le V$ for all $i\in C$;
\item $\alpha_i$ is the inclusion map $U_i \hra V$ for each $i\in C$.
\end{enumerate} 

It is easy to check that $\mbf X\mapsto X_0$ is a ring isomorphism from 
$\End(T_C)$ onto $M^C (q)$. Hence,
\begin{equation}\label{TCeq}
\rho_C (q) = \theta(T_C(q)).
\end{equation}

The following is an easy exercise.

\begin{prop}\label{TC} Let $C$ be a preordered set. Then
the rings $\End(T_C^*)$ and $M^{C^*}(q)$ are isomorphic. Hence,
$\rho_{C^*}(q) = \theta(T_C^*)$.
\end{prop}

Combining~\eqref{TCeq}, Proposition~\ref{dualquiv} and Proposition~\ref{TC}, we 
obtain the following result.

\begin{cor}\label{dualpre} For any finite preordered set $C$, for all prime
powers $q$,
$$
\rho_C (q) = \rho_{C^*}(q).
$$
In particular, if $(l_1,\ldots,l_s)$ is a tuple of nonnegative integers,
then 
$$
\rho_{(l_1,\ldots,l_s)}(q) = \rho_{(l_s,\ldots,l_1)}(q).
$$
\end{cor}

The following table gives the values of $\rho_{\mbf l}(q)$ for all tuples 
$\mbf l=(l_1,\ldots,l_s)$ with $l_1+\cdots+l_s\le 6$ (see~\cite[Appendix B]{thesis} for a detailed computation). In view of 
Corollary~\ref{dualpre}, we list only one representative for each pair 
of distinct tuples $(l_1,\ldots,l_s)$ and $(l_s,\ldots,l_1)$.

\begin{longtable}{|c|c|}
\hline
$\mbf l$ & $\rho_{\mbf l}(q)$ \\
\hline\endhead
\multicolumn{2}{c}{\textit{Continued on the next page}}
\endfoot
\endlastfoot
(1) & 1 \\
\hline
(2) & 2 \\
\hline
(1,1) & 2 \\
\hline
(3) & 3 \\
\hline
(1,2) & 4 \\
\hline
$(1,1,1)$ 
& 5 \\
\hline
(4) & 5 \\
\hline
(1,3) & 7 \\
\hline
(2,2) & 10 \\
\hline
$(1,1,2)$ 
& 12 \\ 
\hline 
$(1,2,1)$ 
& 11 \\
\hline
$(1,1,1,1)$ 
& 16 \\
\hline
(5) & 7 \\
\hline
(1,4) & 12 \\
\hline
(2,3) & 18 \\
\hline
$(1,2,2)$ 
& 30 \\
\hline
$(2,1,2)$ 
& 31 \\
\hline
$(1,1,3)$ 
& 23 \\
\hline
$(1,3,1)$ 
& 21 \\
\hline 
$(1,1,1,2)$ 
& 43 \\
\hline
$(1,1,2,1)$ 
& 40 \\
\hline
$(1,1,1,1,1)$ 
& 61 \\
\hline
(6) & 11 \\
\hline
(1,5) & 19 \\
\hline
(2,4) & 34 \\
\hline
(3,3) & 37 \\
\hline
$(1,2,3)$ 
& 63 \\
\hline
$(1,3,2)$ 
& 62 \\
\hline
$(2,1,3)$ 
& 66 \\
\hline
$(1,1,4)$
& 43 \\
\hline
$(1,4,1)$ 
& 38 \\
\hline 
$(2,2,2)$ 
& $q+89$ \\
\hline
$(1,1,1,3)$ 
& 93 \\
\hline
$(1,1,3,1)$ 
& 84 \\
\hline
$(1,1,2,2)$ 
& $q+121$ \\
\hline
$(1,2,1,2)$
& 120 \\
\hline
$(1,2,2,1)$ 
& 113 \\
\hline
$(2,1,1,2)$
& $q+127$ \\
\hline
$(1,1,1,1,2)$ 
& $q+185$ \\
\hline
$(1,1,1,2,1)$ 
& 173 \\
\hline
$(1,1,2,1,1)$ 
& $q+170$ \\
\hline
$(1,1,1,1,1,1)$ 
& $q+273$ \\
\hline
\end{longtable}

\vspace{1.5cm}

\begin{center} \large{\textbf{Appendix}

\textbf{Ranks of partition matrices} 

 Anton Evseev and George Wellen 

}
\end{center}

Let $S$ be a subset of $\mN$. 
If $k\in \mZ$, let $p(k,S)$ be the number of partitions $\lda$ such that 
$|\lda|=k$ and $\lda_i\in S$ for all $i$. Recall that $r(k,S)$ is the number of 
partitions $\lda$ such that $\lda_i\in S$ for all $i$ and, for each $s\in S$, there exists $i$ such that $\lda_i=s$.
 
Fix a positive integer $m$, and let $n=m(m+1)/2$. Let $\mathsf P$ be the matrix whose columns are indexed by non-empty subsets of $[1,m]$, whose rows are indexed by nonnegative integers and whose $(k,S)$ entry is $p(k,S)$. Let $\mathsf R$ be the matrix whose rows are indexed by positive integers $k$ and whose columns are indexed by non-empty subsets $S$ of $[1,m]$ with the $(k,S)$ 
entry equal to $r(k,S)$. Thus, $\mathsf P$ and $\mathsf R$ have 
infinitely many rows and $2^m-1$ columns each. 
Our aim is to find out the ranks of $\mathsf P$ and 
$\mathsf R$ and to find linear relations between rows 
of $\mathsf P$ and $\mathsf R$.

If $S\subseteq [1,m]$, let 
$$P_S(X)=\sum_{k=0}^{\infty} p(k,S) X^k$$
be the generating function of the sequence $(p(k,S))_{s=0}^{\infty}$.
Here, and in the sequel, $X$ is a formal variable, and all expressions involving $X$ are assumed to be elements of the ring $\mQ [[X]]$ of formal power series over $\mQ$.  

Observe that 
$$P_{\{i\}}(X)=\sum\limits_{k\geq 0}X^{ik}=\frac{1}{1-X^i}.$$
It follows that, for any non-empty finite subset $S\subset \mN$,
\begin{equation}\label{GFeq2}
P_S(X)=\frac{1}{\prod_{i\in S}(1-X^i)}
\end{equation}
Also, by the inclusion-exclusion formula, 
for all $S\subseteq [1,m]$ and all $k\in\mN$,
\begin{equation}\label{inex}
r(k,S)=\sum_{S'\subseteq S, S\ne\varnothing} (-1)^{|S|-|S'|} p(k,S).
\end{equation}
(Note that $p(k,\varnothing)=0$ for $k>0$.) The following result can be easily proved by induction.
\begin{alem}\label{sums}
For every natural number $k\in [1,n]$ there exists a
non-empty subset $S_k$ of $[1,m]$ such that
$k=\sum\limits_{i\in S_k} i$.
\end{alem}

Let
$$\Delta(X)=\prod_{i=1}^m (1-X^i).$$
Define integers $c_0,c_1,\ldots,c_n$ by 
$$\Delta(X)=\sum_{i=0}^n c_i^{(m)} X^i. $$  
\begin{athm}\label{prank} $\rank (\mathsf P)=n=m(m+1)/2$. 
Moreover, for all $k\ge n$ and all non-empty $S\subseteq [1,m]$,
$$p(k,S)=-\sum_{j=k-n}^{k-1} c_{k-j} p(j,S).$$ 
\end{athm}
\begin{proof} Let $S$ be a non-empty subset of $[1,m]$.
By \eqref{GFeq2}, 
\begin{eqnarray*}
P_S(X) \Delta(X) & = & \prod_{i\in [1,m]\setminus S} (1-X^i), \quad \text{hence}, 
\label{GFeq}\\ 
\sum_{k=0}^{\infty} \sum_{j=\max(0,k-n)}^k c_{k-j} p(j,S) X^k  & = & 
\prod_{i\in [1,m]\setminus S} (1-X^i). 
\end{eqnarray*}
If $k\ge n$, the coefficient in $X^k$ on the right-hand side is $0$ (since 
$S\ne\varnothing$), so
$$\sum_{j=k-n}^k c_{k-j} p(j,S) = 0.$$
Then the expression for $p(k,S)$ follows from the fact that $c_0=1$. It follows that $\rank(\mathsf P)\le n$. 

To show that $\rank(\mathsf P)\ge n$, 
it is enough to prove that the polynomials 
$P_S(X)\Delta(X)$ span a subspace of dimension at least $n$ in $\mQ [[X]]$ as 
$S$ varies among the non-empty subsets of $[1,m]$. By~\eqref{GFeq}, 
$$\deg(P_S\cdot \Delta)=\sum\limits_{i\in [1,m]\setminus S} i
 = n-\sum\limits_{i\in S} i.$$
Hence, by Lemma \ref{sums}, for each integer 
$k\in [0,n-1]$, there exists a non-empty $S\subseteq [1,m]$ such that 
 $\deg(P_S \cdot \Delta)=k$. The result follows. \end{proof}    

\begin{acor}\label{rrank} $\rank (\mathsf R)=n$. Moreover, for all $k>n$ and all non-empty $S\subseteq [1,m]$,
$$r(k,S)=-\sum_{j=k-n}^{k-1} c_{k-j} r(j,S).$$
\end{acor}
\begin{proof} The expression for $r(k,S)$ follows from Theorem \ref{prank} and 
\eqref{inex}. By \eqref{inex}, $\rank(\mathsf R)$ is equal to the rank of the matrix obtained by removing the $0$-row from $\mathsf P$. By Theorem \ref{prank},
$$c_n p(0,S) = - \sum_{j=1}^n c_{n-j} p(j,S),$$ 
so the $0$-row of $\mathsf P$ is a linear combination of the next $n$ rows 
($c_n\ne 0$). Hence, $\rank(\mathsf R)=n$. \end{proof}

Define $d_j=\sum_{i=0}^{n-j} c_i$ for $j=1,2,\ldots n$.
\begin{aprop}\label{dj} For all non-empty $S\subseteq [1,m]$,
$$\sum_{j=1}^n d_j r(j,S) = \sum_{j=1}^n d_j p(j,S) = 
\begin{cases}
1 & \text{if } S=[1,m], \\ 
0 & \text{otherwise.}
\end{cases}
$$
\end{aprop}

\begin{proof} Let $f(X)=(P_S (X) - 1) \Delta(X)$. Then
$$
f(X) =  \prod_{i\in [1,m]\setminus S} (1-X^i)-\Delta(X).
$$  
Thus, $f(X)$ is a polynomial of degree $n$, and $f(1)$ is the sum of the coefficients of $f(X)$ in $X^0,X^1,\ldots,X^n$.
On the other hand,
$$f(X)=\sum_{j=1}^{\infty} \sum_{i=0}^n c_i p(j,S) X^{i+j}.$$
Therefore, 
$$f(1)=\sum_{j=1}^n \sum_{i=0}^{n-j} c_i p(j,S) = \sum_{j=1}^n d_j p(j,S).$$
Hence, 
$$
\sum_{j=1}^n d_j p(j,S)= f(1)= 
\begin{cases}
1 & \text{if } S=[1,m], \\ 
0 & \text{otherwise.}
\end{cases}
$$
By~\eqref{inex}, $\sum_{j=1}^n d_j r(j,S)$ is equal to this too. 
\end{proof}

\bibliographystyle{amsplain}
\bibliography{porc,parabolic}

\end{document}